\documentclass[final,onefignum,onetabnum]{siamart171218}
\usepackage{amsmath}
\usepackage{mathtools}
\usepackage{braket,amsfonts}
\usepackage{array}
\usepackage[caption=false]{subfig}
\usepackage{pgfplots}
\newsiamthm{claim}{Claim}
\newsiamremark{hypothesis}{Hypothesis}
\crefname{hypothesis}{Hypothesis}{Hypotheses}
\usepackage{algorithmic}
\usepackage{graphicx,epstopdf}
\Crefname{ALC@unique}{Line}{Lines}
\usepackage{amsopn}

\usepackage{xspace}
\usepackage{bold-extra}
\usepackage[most]{tcolorbox}

\colorlet{texcscolor}{blue!50!black}
\colorlet{texemcolor}{red!70!black}
\colorlet{texpreamble}{red!70!black}
\colorlet{codebackground}{black!25!white!25}

\lstdefinestyle{siamlatex}{%
  style=tcblatex,
  texcsstyle=*\color{texcscolor},
  texcsstyle=[2]\color{texemcolor},
  keywordstyle=[2]\color{texemcolor},
  moretexcs={cref,Cref,maketitle,mathcal,text,headers,email,url},
}

\tcbset{%
  colframe=black!75!white!75,
  coltitle=white,
  colback=codebackground, 
  colbacklower=white, 
  fonttitle=\bfseries,
  arc=0pt,outer arc=0pt,
  top=1pt,bottom=1pt,left=1mm,right=1mm,middle=1mm,boxsep=1mm,
  leftrule=0.3mm,rightrule=0.3mm,toprule=0.3mm,bottomrule=0.3mm,
  listing options={style=siamlatex}
}

\newtcblisting[use counter=example]{example}[2][]{%
  title={Example~\thetcbcounter: #2},#1}

\newtcbinputlisting[use counter=example]{\examplefile}[3][]{%
  title={Example~\thetcbcounter: #2},listing file={#3},#1}

\DeclareTotalTCBox{\code}{ v O{} }
{ 
  fontupper=\ttfamily\color{black},
  nobeforeafter,
  tcbox raise base,
  colback=codebackground,colframe=white,
  top=0pt,bottom=0pt,left=0mm,right=0mm,
  leftrule=0pt,rightrule=0pt,toprule=0mm,bottomrule=0mm,
  boxsep=0.5mm,
  #2}{#1}

\patchcmd\newpage{\vfil}{}{}{}
\usepackage{tikz}
\usetikzlibrary{calc}
\usetikzlibrary{fit}

\usepackage{amssymb}
\usepackage{enumerate}
\usepackage{enumitem}
\usepackage{mathtools}
\allowdisplaybreaks

\newcommand{\red}{\color{black}}

\newcommand{\mR}{\mathbb{R}}

\newcommand{\mC}{\mathcal{C}}
\newcommand{\mL}{\mathcal{L}}

\def\I{\mathcal{I}}
\def\P{{\mathcal P}}
\def\Q{{\mathcal Q}}

\def\T{{\mathcal T}}
\def\0{{\boldsymbol 0}}
\def\ba{{\boldsymbol{a}}}
\def\bb{{\boldsymbol{b}}}

\newcommand{\diag}{\text{diag}}
\newcommand{\tUp}{\tilde{\Upsilon}}

\def\by{{\boldsymbol{y}}}

\def\bw{{\boldsymbol{w}}}
\def\bq{{\boldsymbol{q}}}
\def\bp{{\boldsymbol{p}}}
\def\bx{{\boldsymbol{x}}}
\def\be{{\boldsymbol{e}}}
\def\bu{{\boldsymbol{u}}}
\def\bd{{\boldsymbol{d}}}
\def\bl{{\boldsymbol{l}}}
\def\bz{{\boldsymbol{z}}}
\def\blambda{{\boldsymbol{\lambda}}}

\allowdisplaybreaks[1]

\title{Exponential Decay in the Sensitivity Analysis of Nonlinear Dynamic Programming\thanks{Preprint ANL/MCS-P9182-0520, Mathematics and Computer Science Division, Argonne National Laboratory}}

\author{Sen Na\thanks{Department of Statistics, The University of Chicago, Chicago, IL (\email{senna@uchicago.edu}).}
\and Mihai Anitescu\thanks{Mathematics and Computer Science Division, Argonne National Laboratory, Lemont, IL (\email{anitescu@mcs.anl.gov}).}
}

\begin{document}

\maketitle
	
\begin{abstract}
In this paper, we study the sensitivity of discrete-time dynamic programs with nonlinear dynamics and objective to perturbations in the initial conditions and reference parameters. Under uniform controllability and boundedness assumptions for the problem data, we prove that the directional derivative of the optimal state and control at time $k$, $\bx^*_{k}$ and $\bu^*_{k}$, with respect to the reference signal at time $i$, $\bd_{i}$, will have exponential decay in terms of $|k-i|$ with a decay rate~$\rho$ independent of the temporal horizon length. The key technical step is to prove that a version of the convexification approach proposed by Verschueren et al. can be applied to the KKT conditions and results in a convex quadratic program with uniformly bounded data. In turn, Riccati techniques can be further employed to obtain the sensitivity result, borne from the observation that the directional derivatives are solutions of quadratic programs with structure  similar to the KKT conditions themselves. We validate our findings with numerical experiments on a small nonlinear, nonconvex, dynamic program. 

\end{abstract}
	

\section{Introduction}\label{sec:1}

We consider the following discrete-time nonlinear dynamic programming (NLDP) problem with recursive equality constraints:
\begin{subequations}\label{pro:1}
\begin{align}
\min _{\bx,\bu}\text{\ \ } &\sum_{k=0}^{N-1} g_k(\bx_k,\bu_k,\bd_k) + g_N(\bx_N),  \label{pro:1a}\\
\text{s.t.} \text{\ \ } &\bx_{k+1} = f_k(\bx_k, \bu_k,\bd_k),\ \ \ \ k = 0, 1, \ldots, N-1,  \label{pro:1b}\\
&\bx_0 = \bar{\bx}_0. \label{pro:1c}
\end{align}
\end{subequations}
Here $\bx_k\in\mR^{n_x}$ is the state variable, $\bu_k\in\mR^{n_u}$ is the control variable, and $\bd_k\in\mR^{n_d}$~is {\red a reference variable which is externally provided data;} $g_k:\mR^{n_x}\times\mR^{n_u}\times\mR^{n_d}\rightarrow\mR$, $f_k:\mR^{n_x}\times\mR^{n_u}\times\mR^{n_d}\rightarrow\mR^{n_{x}}$
are~the cost function and constraint function for stage $k$, respectively; $N$ is the temporal horizon length; and $\bar{\bx}_0$ is the initial state variable. We also define $\bd_{-1} = \bar{\bx}_0$ for simplicity. In this {\red paper}, {\red we let $\bd = (\bd_{-1}; \bd_0; \ldots;\bd_{N-1})$ and} assume $g_k, f_k$ are both twice continuously differentiable.

Problem (\ref{pro:1}) includes a wide range of classical dynamic programs (DPs) as special cases. In our context, we call it quadratic dynamic programming (QDP) whenever, for any stage $k$, $g_k$ is quadratic and $f_k$ is affine. Since the reference vector is a parameter to the problem, we treat it as the input and aim to (locally) analyze the dependence of the optimal state and control variable on it. Specifically, we are interested in answering the following question:
\begin{center}
\textit{For any two stages $k$ and $i$, how much do optimal $\bx_k$ and $\bu_k$ change \\ when we perturb $\bd_i$?}
\end{center}
We can formulate this question using directional derivatives. Suppose we perturb $\bd$ along a parameterized path:
\begin{align}\label{equ:5}
\bd(\epsilon) = \bd^0 + \epsilon\bl + o(\epsilon),
\end{align}
where $\epsilon$ is the parameter of the path, $\bd^0$ is the unperturbed variable, and the unit~vector $\bl$ indicates the perturbation direction. Based on the path (\ref{equ:5}), we can define the following two directional derivatives (if the limits exist):
\begin{align}\label{equ:6}
\bp_k  &= \lim\limits_{\epsilon\searrow0}\frac{\bx_k^*(\bd(\epsilon))-\bx_k^*(\bd(0))}{\epsilon}, \text{\ \ } \bq_k  = \lim\limits_{\epsilon\searrow0}\frac{\bu_k^*(\bd(\epsilon))-\bu_k^*(\bd(0))}{\epsilon},
\end{align}
where $\bx^*_k(\bd), \bu^*_k(\bd)$ are the optimal solution of Problem (\ref{pro:1}) with specified input~$\bd$. For conciseness, we ignore the dependence of $\bl$ in $\bp_k, \bq_k$ in (\ref{equ:6}). Suppose $\be_q$, $q=1,2,\ldots,Nn_d+n_x$, to be the $q$th canonical basis vector in $\mR^{n_x + Nn_d}$. Then, if we let $\bl = \be_{n_x + i\cdot n_d + j} $, $0 \leq i \leq N-1,\; 1 \leq j \leq n_d$, our perturbation occurs only at the temporal stage $i$, in entry $j$; and, particularly, {\red $\{\be_j\}_{j= 1}^{n_x}$} corresponds to the perturbation at the initial constraints. According to \eqref{equ:6} we locally have $\|\bx_k(\bd(\epsilon)) - \bx_k(\bd(0))\|\approx \epsilon\|\bp_k\|$ (the same for $\bu_k$). With this perturbation structure,  analyzing the magnitude of $\bp_k$ and $\bq_k$ measures the effect on stage $k$ caused by the perturbations at stage $i$.

Some results have been proposed before for this type of local sensitivity analysis. For example, when $g_k(\bx_k, \bu_k, \bd_k) = (\bx_k - \bd_k)^TQ_k(\bx_k - \bd_k) + \bu_k^TR_k\bu_k$ with $Q_k,R_k \succ 0, $ and $f_k$ is an affine function depending only on $(\bx_k, \bu_k)$, Xu and Anitescu \cite{Xu2018Exponentially} showed that the magnitude of  $\bp_k$ and $\bq_k$ for this special case of QDP decays exponentially in $|k - i|$. This observation was crucial for designing a temporal decomposition approach for long-horizon linear-quadratic dynamic programming.
 A similar property, asymptotic decay, was assumed for the nonlinear case in \cite{shin2019parallel} and also resulted in the ability to decompose long-horizon nonlinear dynamic programs and enable parallel computation. It also played a critical role in the analysis of the convergence of model predictive control and, in particular, characterizing its operational horizon \cite{Xu2017Exponentially}. The main objective of this research note is to extend these results to the case of nonconvex QDP and, more generally, nonconvex NLDP.

Dealing with indefinite quadratic matrices in QPs is an important issue in general NLP. Many high-performance solvers are built assuming the positive definiteness of the Hessian matrix, such as the QP solvers HPMPC~\cite{Frison2014High} and qpOASES~\cite{Ferreau2014qpOASES} and the NLP solver FORCES~\cite{Domahidi2012Efficient}. The usual way to accommodate the indefiniteness is to perturb it to a positive matrix, to make sure the calculated step is indeed a descent direction. For example, Levenberg and Marquardt added a multiple of the identity matrix to the Hessian (see \cite{More1978Levenberg, Hanke1997regularizing}) and showed that under some conditions it can achieve superlinear convergence for unconstrained NLP. Nocedal and Wright  \cite{Nocedal2006Numerical} described several methods that directly modify the factors of the Cholesky factorization or Bunch-Parlett-Kauffman factorization for indefinite Hessians. For the constrained case, regularization approaches targeting the KKT matrix have been proposed and applied in interior point methods such as IPOPT~\cite{Waechter2005implementation} and its variants \cite{Zavala2008Interior, Biegler2009Large, Wan2016Structured}. Recently, targeting QDPs stemming from NLDP, Verschueren et al. proposed  transforming the indefinite Hessian to a positive definite matrix, at a point where the second-order conditions hold, while \textit{keeping the reduced Hessian unchanged} \cite{Verschueren2017Sparsity}. We call this approach a \textit{convexification procedure}. In \cite{Verschueren2017Sparsity} the researchers showed that this convexification can be conducted independently of the solvers, preserve the block sparsity structure of the problem, and have linear computational complexity with respect to the horizon length.

The convexification procedure from \cite{Verschueren2017Sparsity} plays an important role in this work, although for our ends we will need to enhance it in several directions. First, it depends on a certain regularization parameter whose selection is unclear because the analysis in \cite{Verschueren2017Sparsity} is centered on proving only its existence using a continuity argument. Second, although not crucial for convexification itself, that analysis does not include linear terms in the objective. In the present paper, we enhance their convexification procedure by addressing these limitations. In particular, we modify it to allow for linear terms in the objective, and we  connect the range of the regularization parameter to quantities characterizing the optimality of the primal vector. In turn these enhancements allow us to prove that, when the system is controllable, the convexification results in the transformed data, including the linear terms, being bounded independent of the horizon. As a result, as was also the goal in \cite{Verschueren2017Sparsity}, we can now apply a Riccati recursion to solve the QDP. Moreover, this recursion is exponentially stable, and we can then extend some of the results in \cite{Xu2018Exponentially} to prove the exponential decay of the sensitivity in the NLDP  case, which is far more general than the QDP case studied in \cite{Xu2018Exponentially}.

\vskip 4pt
\noindent{\bf Notations:} Throughout the paper, we use bold symbols to denote column vectors. We use $I$ to denote the identity matrix whose dimension is always clear from the context. Given an integer $k$, we define $[k] = \{0, 1, \ldots, k\}$ to be all integer indices from 0 to $k$. We let $i\wedge j = \min(i,j)$ {\red and $i\vee j = \max(i,j)$} when $i, j$ are two scalars. For any~vectors $\ba$ and $\bb$, we use $(\ba;\bb)$ to denote a long column vector that stacks them together. We also have $\text{supp}(\ba)$ to be the support set of $\ba$. Without subscript, $\bx, \bu, \bd, \bp$, and $\bq$ denote the corresponding variable for the entire horizon. We also define $\bz = (\bx_0; \bu_0; \ldots; \bx_{N-1}; \bu_{N-1}; \bx_{N})$ to be the whole decision variable ordered by stages. But we may also interpret $\bz = (\bx, \bu)$ when connecting it to the state $\bx$ and control $\bu$ vectors. Let $n_z = (N+1)n_x + Nn_{u}$ be the dimension of $\bz$. Similarly, we define $\bw = (\bp_0; \bq_0; \ldots; \bp_{N-1};\bq_{N-1}; \bp_N)\in\mR^{n_z}$. Given a sequence of matrices $A_i$, we let $\diag(A_1, A_2, \ldots)$ be the block diagonal matrix, whose blocks from upper left to bottom right are specified by $A_1, A_2,\ldots$. We also define $\prod_{i=m}^{n}A_i = A_nA_{n-1}...A_m$ if $n\geq m$ and $I$ otherwise. Similarly, we let $\sum_{i=m}^{n}A_i = 0$ if $n<m$. For a positive definite matrix $A$ and any vector $\bx$, we define $\|\bx\|_A = \sqrt{\bx^TA\bx}$. We also use $\|\cdot\|$ to denote the $l_2$ norm for vector or operator norm for matrix. For the dynamics function $f_k(\bx_k,\bu_k,\bd_k)$, we let $\nabla_{\bu_k}f_k = \frac{\partial f_k}{\partial \bu_k}\in\mR^{n_u\times n_x}$.

\vskip 4pt
\noindent{\bf Structure of the paper:} In section \ref{sec:2}, we introduce some preliminaries and derive the indefinite QP whose solution is the directional derivatives $(\bp, \bq)$. In section \ref{sec:4}, we show the details of {\red the} convexification procedure and propose a method for setting the regularization parameter. In section \ref{sec:5}, we probe the properties of the outputs of the algorithm, based on which in section \ref{sec:6} we then theoretically prove the exponential decay. In section \ref{sec:7} we describe our numerical simulations.  Conclusions plus some potential future work are discussed in section \ref{sec:8}.

\section{Preliminaries}\label{sec:2}

Define the Lagrange function of Problem (\ref{pro:1}) as
\begin{multline}\label{equ:1}
\mL(\bx,\bu,\blambda;\bd) = \sum_{k=0}^{N-1}g_k(\bx_k,\bu_k,\bd_k)+g_N(\bx_N)\\
+\sum_{k=0}^{N-1}\blambda_{k}^T(\bx_{k+1}-f_k(\bx_k,\bu_k,\bd_k)) 
+\blambda_{-1}^T(\bx_0-\bd_{-1}),
\end{multline}
where $\blambda = (\blambda_{-1}; \blambda_0; \ldots;\blambda_{N-1})$ is the Lagrange multiplier with $\blambda_k\in\mR^{n_x}$, $k\in {\red \{-1\}}\cup [N-1]$. 
An {\red immediate} but important observation {\red we justify below} is that any feasible point of Problem (\ref{pro:1}) satisfies the linear independence constraint qualification (LICQ). 
We now define several quantities of interest.

\begin{definition}\label{def:1}

We introduce the notations $\tilde{f}_k(\bz,\bd) = \bx_{k+1} - f_k(\bx_k,\bu_k,\bd_k), \forall k\in[N-1]$ and $\tilde{f}_{-1}(\bz,\bd) = \bx_0 - \bd_{-1}$. For any prespecified $\bd$, we define the Jacobian matrix as follows:
\begin{align*}
G(\bz;\bd) = \big(\nabla_{\bz}\tilde{f}_{-1}(\bz,\bd), \nabla_{\bz}\tilde{f}_{0}(\bz,\bd),\ldots, \nabla_{\bz}\tilde{f}_{N-1}(\bz,\bd)\big)^T\in\mR^{(N+1)n_x\times n_z}.	
\end{align*}
(Note that {\red $G(\bz;\bd)$ } has full row rank {\red because $\frac{\partial  \tilde{f}_k}{\partial x_{k+1}} = I_{n_x}, \forall -1 \leq k \leq N-1$ and  $\frac{\partial  \tilde{f}_k}{\partial x_{j}} = 0, \forall -1 \leq k < j-1 \leq N-1$; this ``staircase" structure ending with identity matrices on each row implies the full row rank property of $G(\bz;\bd)$ and, thus, LICQ holds.}) When $\bz$ is a local solution of Problem (\ref{pro:1}), the critical cone\footnote{The detailed definition is in \cite{Nocedal2006Numerical}. See equation (12.53).} at $\bz$ is the null space of $G(\bz; \bd)$, denoted as $\mC(\bz; \bd)$. We use $Z(\bz;\bd)\in\mR^{n_z\times Nn_u}$ to denote a full column rank matrix whose columns are orthonormal and span the space $\mC(\bz;\bd)$. We denote the Hessian matrix of Lagrangian (\ref{equ:1}) by $\forall k\in[N-1]$
\begin{align*}
H_k =& \begin{pmatrix*}
Q_k & S_k^T\\
S_k & R_k
\end{pmatrix*} = \begin{pmatrix*}
\nabla^2_{\bx_k}\mL & \nabla^2_{\bx_k\bu_k}\mL\\
\nabla^2_{\bu_k\bx_k}\mL & \nabla^2_{\bu_k}\mL
\end{pmatrix*},\text{\ \ } D_k = (D_{k1}, D_{k2}) = (\nabla_{\bd_k\bx_k}^2\mL, \nabla_{\bd_k\bu_k}^2\mL),
\end{align*}
and $H_N = Q_N = \nabla_{\bx_N}^2g_N(\bx_N)$. We let $H = \diag(H_0,...,H_N)$, $D = \diag(D_0,...,D_{N-1})$.

\end{definition}

We now set the stage for the sensitivity analysis. Suppose the unperturbed reference variable is $\bd^0$ and the unperturbed primal-dual optimal solution at $\bd = \bd^0$ is $(\bz^0, \blambda^0; \bd^0) = (\bx^0, \bu^0, \blambda^0; \bd^0)$. In what follows, all matrices from Definition \ref{def:1} are evaluated at this primal-dual solution, and their dependence on the primal-dual point is suppressed. We require the following second-order sufficient condition (SOSC).

\begin{assumption}[SOSC]\label{ass:1-2}
At $(\bx^0,\bu^0,\blambda^0; \bd^0)$, the reduced Hessian of Problem (\ref{pro:1}) is positive definite. That is,
\begin{align*}
Z^THZ\succ 0.
\end{align*}

\end{assumption}

We mention that Assumption \ref{ass:1-2} is standard in sensitivity analysis \cite{Bonnans2000Perturbation}. It implies that $\bz^0 = (\bx^0, \bu^0)$ is a strict local solution and guarantees that directional derivatives of the solution with respect to the parameter $\bd$ are well defined. Under SOSC, we have the following sensitivity analysis theorem.

\begin{theorem}[Sensitivity analysis of Problem (\ref{pro:1})]\label{thm:3}
Suppose we perturb $\bd$~along the path in (\ref{equ:5}) for Problem (\ref{pro:1}). Then, under Assumption \ref{ass:1-2}, the directional derivative $\bw = (\bp, \bq)$ of the solution $\big(\bx^*(\bd(\epsilon)),\bu^*(\bd(\epsilon))\big)$ for positive $\epsilon$ at $0$, defined in (\ref{equ:6}), exists and is the solution of the following quadratic program:
\begin{subequations}\label{equ:9}
\begin{align}
\Q\P:\text{\ \ \ }\min_{\bw}\text{\ \ }&\bw^TH\bw + 2\bl^TD\bw, \label{equ:9a}\\
\text{s.t.}\text{\ \ } &G\bw = \by. \label{equ:9b}
\end{align}
\end{subequations}
Here, we denote {\red $\by = (\by_{-1}; \by_{0}; \ldots; \by_{N-1})\in\mR^{(N+1)n_x}$ with $\by_{-1} = \bl_{-1}$ and $\by_k = (\nabla_{\bd_k}f_k)^T\bl_k$ for $ k\in[N-1]$.} The matrices $H$, $D$, $G$ are from Definition \ref{def:1} and evaluated at $(\bz^0, \blambda^0;\bd^0)$.
	
\end{theorem}

\begin{proof}

See Theorem 5.53 and Remark 5.55 in \cite{Bonnans2000Perturbation}. Note that Problem (\ref{pro:1}) satisfies Gollan's regularity condition at $(\bz^0, \bd^0)$ in any direction (as  implied by LICQ). {\red See also \cite{Robinson1976First, Robinson1982Generalized, Robinson1987Local} and references therein for detailed introduction of sensitivity analysis of nonlinear programming.}
\end{proof}

Under Assumption \ref{ass:1-2}, the quadratic program in (\ref{equ:9}) has a unique global minimizer. Note that constraints (\ref{equ:9b}) come from linearizing (\ref{pro:1b})--(\ref{pro:1c}) whereas the objective (\ref{equ:9a}) is a quadratic approximation of the Lagrange function (\ref{equ:1}).  When $H$ is indefinite, solving \eqref{equ:9} in the context of Problem \eqref{pro:1}  becomes difficult. For example, Riccati-based approaches are not applicable since they all require positiveness of $H$. As shown in \cite{Verschueren2017Sparsity}, however, we can convexify (\ref{equ:9}), \textit{without changing its solution space,} an endeavor we aim to refine in this work.

\section{Convexification procedure with linear shifting}\label{sec:4}

We write Problem (\ref{equ:9}) in a way that explicitly connects it to \eqref{pro:1}. We use $A_k, B_k, C_k$ to denote $\nabla_{\bx_k}^Tf_k$, $\nabla_{\bu_k}^Tf_k$, $\nabla^T_{\bd_k}f_k$ evaluated at $(\bz^0; \bd^0)$. Problem (\ref{equ:9}) can then be written as
\begin{subequations}\label{equ:11}
\begin{align}
\min_{\bp,\bq}\text{\ \ } &\sum_{k=0}^{N-1} \begin{pmatrix*}
\bp_k\\
\bq_k\\
\bl_k
\end{pmatrix*}^T \begin{pmatrix*}
Q_k & S_k^T & D_{k1}^T\\
S_k & R_k & D_{k2}^T\\
D_{k1} & D_{k2} & 0
\end{pmatrix*}\begin{pmatrix*}
\bp_k\\
\bq_k\\
\bl_k
\end{pmatrix*} + \bp_N^T Q_N\bp_N, \label{equ:11a}\\
\text{s.t.}\text{\ \ } &\bp_{k+1} = A_k\bp_k+B_k\bq_k+C_k\bl_k,\text{\ \ \ \ } \forall k\in[N-1], \label{equ:11b}\\
&\bp_{0} = \bl_{-1}. \label{equ:11c}
\end{align}
\end{subequations}

If $C_k =0$ and $D_k = 0$, one can directly use the convexification procedure from~\cite{Verschueren2017Sparsity}. Since those terms are nonzero in our work, we need to slightly extend that procedure. To this end, we start by introducing the following uniform (in $N$) SOSC assumption.

\begin{assumption}[Uniform SOSC]\label{ass:2}
We assume there exists $\gamma>0$ independent from $N$ such that the reduced Hessian of Problem (\ref{pro:1}), at its unperturbed solution $(\bz^0, \blambda^0; \bd^0)$, satisfies
\begin{align}\label{equ:8}
Z^THZ\succeq \gamma I.
\end{align}

\end{assumption}

Comparing Assumption \ref{ass:2} with Assumption \ref{ass:1-2}, we see that Assumption \ref{ass:2} requires the reduced Hessian matrix to have a positive lower bound uniformly with horizon length $N$. While our convexification results will also hold under general SOSC, this stronger SOSC will indicate that sensitivity to perturbation decays at a rate independent of $N$, which offers a theoretical guarantee on sensitivity analysis of long-horizon nonconvex NLDP. We note that Assumption \ref{ass:2} holds for a wide range of dynamics because of the natural block structures in them. For example, the QDP studied in \cite{Xu2018Exponentially}, which assumed $H_k = \diag(Q_k, R_k)\succeq \gamma I$, satisfies (\ref{equ:8}). A detailed discussion is provided in the next remark, and a concrete example is tested in numerical experiments described in section \ref{sec:7}.

{\red We find it difficult to give a characterization of condition \eqref{equ:8} that would convincingly identify how likely it is to occur in practice (though for fixed $N$, being equivalent to SOSC, it is a natural requirement). We give an example of a problem class (in the sense of $N$ being allowed to vary) that exhibits this property. }

\begin{remark}

{\red An example that satisfies Assumption \ref{ass:2}.} Suppose $n_x = n_u = 1$, $A_k = B_k = 1$, $\forall k\in[N-1]$. By Definition~\ref{def:1}, we can formulate the Jacobian matrix $G\in\mR^{(N+1)\times (2N+1)}$, and one of matrices, $\tilde{Z}\in\mR^{(2N+1)\times N}$, that spans its null space can be written as
\[\tilde{Z}^T = \begin{pmatrix*}
0 & 1 & 1 & -1 \\
&&&1 & 1 & -1\\
&&&&&\ddots & \ddots & \ddots\\
&&&&&&&1 & 1 & -1 & 0\\
&&&&&&&&&1 & 1
\end{pmatrix*}.
\]
Note that $\tilde{Z}^T\tilde{Z}$ is a tridiagonal matrix with diagonal entry $3$ (the last one is $2$) and superdiagonal entry $-1$. Thus, by the Gershgorin circle theorem (see \cite{Bauer1960Norms}), we know $\tilde{Z}^T\tilde{Z}$ can be upper bounded uniformly with its dimension $N$. In particular, we have $I\preceq \tilde{Z}^T\tilde{Z}\preceq 4I$. Consider a simple Hessian matrix where $Q_k = a_k$, $R_k = b_k$, $S_k = 0$, for all stages $k$. Then one can easily  see that $\tilde{Z}^TH\tilde{Z}$ is also a tridiagonal matrix, with diagonal entry $b_{k-1}+b_k+a_k$ (the last one is $b_{N-1}+a_N$) and superdiagonal entry $-b_k$. By the Gershgorin circle theorem again, we see that as long as $a_k\geq 2|b_k| + 2|b_{k-1}| + 4\gamma$, $\forall k\in[1, N-1]$, and $a_N\geq 2|b_{N-1}| + 4\gamma$, we can get $\tilde{Z}^TH\tilde{Z}\succeq4\gamma I$, which further implies (\ref{equ:8}), as can be seen from
\begin{align*}
\tilde{Z}^TH\tilde{Z}\succeq 4\gamma I\Longrightarrow\tilde{Z}^TH\tilde{Z}\succeq\gamma\tilde{Z}^T\tilde{Z} =  \gamma Z_2^TZ_2 \Longrightarrow Z_1^THZ_1\succeq \gamma I,
\end{align*}
where $\tilde{Z} = Z_1Z_2$ is the QR decomposition. More generally, even if $n_x, n_u>1$ and $S_k\neq 0$,  $\tilde{Z}^TH\tilde{Z}$ still has a block tridiagonal structure, and Assumption \ref{ass:2} holds whenever $Q_k$ is sufficiently large. We  point out that $R_k$ can be indefinite or negative definite under this setup. Thus, the problem is still nonconvex.

\end{remark}

With Assumption \ref{ass:2}, we now present our convexification procedure. While the algebra is similar to the one in \cite{Verschueren2017Sparsity}, we describe it here for self-consistency of the presentation. 

\begin{definition}\label{def:3}

Given a sequence of matrices $\{\bar{Q}_k\}_{k=0}^{N}\in\mR^{n_x\times n_x}$, we define the quadratic cost sequence as $\bar{r}_k(\bx) = \bx^T\bar{Q}_k\bx$, $\forall k\in[N]$. Also, for $k\in[N-1]$ we denote the $k$th stage cost in objective (\ref{equ:11a}) as $L_k(\bp_k,\bq_k)$, {\red which is a quadratic plus linear function of $\bp_k, \bq_k$,} and $L_N(\bp_N) =  \bp_N^T Q_N\bp_N$.
\end{definition}

The following lemma will show how to recursively use constraints to modify the objective matrices while keeping the {\red primal solution} of the QDP unchanged.  

\begin{lemma}\label{lem:3}

Given any sequence of matrices $\{\bar{Q}_k\}_{k=0}^N\in\mR^{n_x\times n_x}$, we define the transformed cost function as $\bar{L}_N(\bp_N) = L_N(\bp_N) - \bar{r}_N(\bp_N)$ and $\forall k\in[N-1]$
\begin{align*}
\bar{L}_k(\bp_k,\bq_k) = L_k(\bp_k,\bq_k) -\bar{r}_k(\bp_k) + \bar{r}_{k+1}(A_k\bp_k+B_k\bq_k+C_k\bl_k).
\end{align*}
Then replacing the objective $L_k$ in (\ref{equ:11a}) with $\bar{L}_k$ will result in the same solution.

\end{lemma}

\begin{proof}

The argument can be derived from the following observation:
\begin{align*}
\sum_{k=0}^{N-1}\bar{L}_k(\bp_k,& \bq_k) + \bar{L}_N(\bp_N) \\
& =  \sum_{k=0}^{N-1}\bigg(L_k(\bp_k, \bq_k) -\bar{r}_k(\bp_k) + \bar{r}_{k+1}(A_k\bp_k+B_k\bq_k+C_k\bl_k) \bigg)+ \bar{L}_N(\bp_N)\\
& =  \sum_{k=0}^{N-1}\bigg(L_k(\bp_k, \bq_k) -\bar{r}_k(\bp_k) + \bar{r}_{k+1}(\bp_{k+1}) \bigg)+ \bar{L}_N(\bp_N)\\
& =  \sum_{k=0}^{N-1}L_k(\bp_k, \bq_k) -\bar{r}_0(\bp_0) + \bar{r}_{N}(\bp_{N})+ \bar{L}_N(\bp_N)\\ 
& =  \sum_{k=0}^{N-1}L_k(\bp_k, \bq_k)+ L_N(\bp_N) -\bar{r}_0(\bl_{-1}),
\end{align*}
where the second and the last equalities are due to (\ref{equ:11b}) and (\ref{equ:11c}), respectively. Since the objective after the replacement stays the same, the conclusion follows. 
\end{proof}

Let us focus on the quadratic matrix in $\bar{L}_k(\bp_k,\bq_k)$. It can be written as

\begin{align*}
\begin{pmatrix}
\tilde{Q}_k & \tilde{S}^T_k & \tilde{D}_{k1}^T\\
\tilde{S}_k & \tilde{R}_k 
& \tilde{D}_{k2}^T\\
\tilde{D}_{k1} & \tilde{D}_{k2} & *
\end{pmatrix} \coloneqq \begin{pmatrix}
Q_k-\bar{Q}_k& S_k^T & D_{k1}^T\\
S_k & R_k & D_{k2}^T\\
D_{k1} & D_{k2} & 0
\end{pmatrix}+\begin{pmatrix}
A_k^T\\
B_k^T\\
C_k^T
\end{pmatrix}\bar{Q}_{k+1}\begin{pmatrix}
A_k & B_k & C_k
\end{pmatrix}.
\end{align*}
We use ``$*$'' to denote the quadratic matrix for $\bl_k$, which can be seen as constant. Now, the problem of obtaining a convexification transformation translates to finding $\{\bar{Q}_k\}_{k=0}^N$ such that {\red the corner matrix $\tilde{H}_k \coloneqq \begin{pmatrix}
\tilde{Q}_k & \tilde{S}^T_k \\
\tilde{S}_k & \tilde{R}_k 
\end{pmatrix}$} 
is positive definite. Our convexification procedure is shown in Algorithm \ref{alg:1}.

\begin{algorithm}[!h]
	\caption{Convexification Procedure with Linear Shifting}
	\label{alg:1}
	\begin{algorithmic}[1]
		\STATE \textbf{Input:} $\{H_k\}_{k=0}^N$, $\{A_k, B_k, C_k, D_k\}_{k = 0}^{N-1}$, a scalar $\delta$;
		\STATE $\tilde{H}_N = \tilde{Q}_N = \delta I$;
		\STATE $\bar{Q}_N = Q_N-\tilde{Q}_N$;
		\FOR {$k = N-1, \ldots, 0$}
		\STATE $\begin{pmatrix}
		\hat{Q}_k & \tilde{S}_k^T & \tilde{D}_{k1}^T\\
		\tilde{S}_k & \tilde{R}_k & \tilde{D}_{k2}^T\\
		\tilde{D}_{k1} & \tilde{D}_{k2} & *
		\end{pmatrix} = \begin{pmatrix}
		Q_k & S_k^T & D_{k1}^T\\
		S_k & R_k & D_{k2}^T \\
		D_{k1} & D_{k2} & 0
		\end{pmatrix} + \begin{pmatrix}
		A_k^T\\
		B_k^T\\
		C_k^T
		\end{pmatrix}\bar{Q}_{k+1}\begin{pmatrix}
		A_k & B_k & C_k
		\end{pmatrix}$;
		\STATE $\tilde{Q}_k = \tilde{S}_k^T\tilde{R}_k^{-1}\tilde{S}_k+\delta I$;
		\STATE $\tilde{H}_k = \begin{pmatrix}
		\tilde{Q}_k & \tilde{S}_k^T\\
		\tilde{S}_k & \tilde{R}_k
		\end{pmatrix}$;
		\STATE $\bar{Q}_k = \hat{Q}_k - \tilde{Q}_k$;
		\ENDFOR
		\STATE \textbf{Output:} $\{\tilde{H}_k\}_{k = 0}^N$, $\{\tilde{D}_k\}_{k = 0}^{N-1}$.
	\end{algorithmic}
\end{algorithm}

This algorithm is an easy extension of \cite{Verschueren2017Sparsity}, and some of its features (e.g., preserving sparsity structure) have been described in \cite{Verschueren2017Sparsity}. We will use $\tilde{H}(\delta)$ and $\tilde{D}(\delta)$ to denote the output $\tilde{H},\tilde{D}$ of the algorithm when choosing a specific $\delta$.

From Lemma \ref{lem:3} we know that replacing $H_k, D_k$ in (\ref{equ:11a}) with $\tilde{H}_k(\delta), \tilde{D}_k(\delta)$ for any $\delta$ will result in the same optimal solution; we subsequently aim to choose $\delta$ so that $\tilde{H}_k(\delta)$ is uniformly positive definite. In \cite{Verschueren2017Sparsity} such a $\delta$ was proved to exist but not otherwise estimated in terms of other elements of the problem. Because of the continuity argument used, it cannot be directly assessed whether $\tilde{H}_k(\delta)$ would grow or degrade its positive definiteness should $N$ be increased. 
Since $\delta$ is of special interest in our paper, we propose a finer analysis of it by combining Algorithm \ref{alg:1} with a Riccati recursion. Further, as we later show, the uniform boundedness of $\tilde{H}_k(\delta), \tilde{D}_k(\delta)$ with feasible $\delta$ is implied by a controllability condition. In the remainder of this section, we first determine a \textit{sufficient} interval for $\delta$. Here \textit{sufficient} means ensuring the positive definiteness (and thus, invertibility) of $\tilde{R}(\delta)$ and $\tilde{H}_k(\delta)$; so that the recursion in Algorithm \ref{alg:1} is well defined and the equivalent QDP is strongly convex.

The following two lemmas show the relationship between the convexification algorithm and the solution of Problem (\ref{equ:11}).

\begin{lemma}[Cost-to-go function for QDP with linear shifting]\label{lem:4} 	
For Problem~(\ref{equ:11}), let $K_N = Q_N$ {\red be the terminal cost-to-go matrix}. For any $k\in[N-1]$, we define $W_k = R_k+B_k^TK_{k+1}B_k$, {\red $k$th cost-to-go matrix} 
\begin{align}\label{equ:lem4:1}
K_k = Q_k+A_k^TK_{k+1}A_k - (B_k^TK_{k+1}A_k+S_k)^TW_k^{-1}(B_k^TK_{k+1}A_k+S_k),
\end{align}
{\red feedback matrix} $P_k = -W_k^{-1}(B_k^TK_{k+1}A_k+S_k)$, and {\red closed-loop state transition matrix} $E_k = A_k+B_kP_k$. Further, $\forall i, k\in[N-1]$, we let $V_i^k  = -K_{i+1}\prod_{j=k}^{i}E_j$ and $M_i^k = -(D_{i1}+D_{i2}P_i)\prod_{j = k}^{i-1}E_j$. Then under Assumption~\ref{ass:1-2}, we have the following:
\begin{enumerate}[label=(\roman*),topsep=-13pt,leftmargin=15pt]
\setlength\itemsep{-0.1em}
\item $W_k\succ 0$, $\forall k\in[N-1]$;
\item The optimal control variable at stage $k$ for $k\in[N-1]$ is
\end{enumerate}
\begin{multline}\label{equ:12}
\noindent\bq_k(\bp_k) = P_k\bp_k + W_k^{-1}B_k^T\sum_{i = k+1}^{N-1} (M_i^{k+1})^T\bl_i + W_k^{-1}B_k^T\sum_{i=k+1}^{N-1}(V_{i}^{k+1})^TC_i\bl_i\\
- W_k^{-1}B_k^TK_{k+1}C_k\bl_k-W_k^{-1}D_{k2}^T\bl_k;
\end{multline}
\begin{enumerate}[label=and (\roman*),topsep=-5pt,leftmargin=40pt,]
\setlength\itemsep{-0.1em}
\setcounter{enumi}{2}
\item the cost-to-go function is
\end{enumerate}
\begin{align}\label{equ:13}
J_k(\bp_k)& = \bp_k K_k\bp_k -2\sum_{i = k}^{N-1}\bl_i^TM_i^k\bp_k - 2\sum_{i = k}^{N-1}\bl_i^TC_i^TV_i^k\bp_k + T_k, \ \ \ \forall k\in[N],
\end{align}
where $T_N = 0$ and $T_k$, depending on $\{\bl_i\}_{i\geq k}$, satisfies $T_k = 0$ if $\bl_i = 0,\; \forall k\leq i\leq N-1$.

\end{lemma}

\begin{proof}

We will use reverse induction to prove property (iii) and incidentally~prove (i) and (ii) together. We know $J_N(\bp_N) = \bp_N^TQ_N\bp_N = \bp_N^TK_N\bp_N$. Suppose $J_{k+1}(\bp_{k+1})$ satisfies equation (\ref{equ:13}). At stage $k$, by definition of the cost-to-go function, we have
\begin{align}\label{pequ:3}
J_k(\bp_k)= &\min_{\bq_k}\bigg\{ L_k(\bp_k,\bq_k) + J_{k+1}(A_k\bp_k+B_k\bq_k+C_k\bl_k)\bigg\}\nonumber\\
= &\bp_k^TQ_k\bp_k + 2\bl_k^TD_{k1}\bp_k + \min_{\bq_k}\bigg\{\bq_k^TR_k\bq_k + 2\bp_k^TS_k^T\bq_k + 2\bl_k^TD_{k2}\bq_k \\
&+ J_{k+1}(A_k\bp_k+B_k\bq_k+C_k\bl_k)\bigg\}. \nonumber
\end{align}
Under Assumption \ref{ass:1-2}, we know Problem (\ref{equ:11}) has a unique solution. So the minimization in (\ref{pequ:3}) has a unique minimizer, whereas the quadratic matrix of its objective is $W_k$. Thus, we must have $W_k\succ 0$ (proving the induction step for (i)). In particular, for (\ref{pequ:3}) we plug in the formula of $J_{k+1}(\cdot)$, shown in (\ref{equ:13}) at $k+1$, and have
\begin{align}
J_k(\bp_k) = & \bp_k^TQ_k\bp_k + (A_k\bp_k+C_k\bl_k)^TK_{k+1}(A_k\bp_k+C_k\bl_k)+2\bl_k^TD_{k1}\bp_k \nonumber\\
&-2\sum_{i=k+1}^{N-1}\bl_i^TM_i^{k+1}(A_k\bp_k+C_k\bl_k)-2\sum_{i=k+1}^{N-1}\bl_i^TC_i^TV_i^{k+1}(A_k\bp_k+C_k\bl_k)+T_{k+1} \nonumber\\
&+ \min_{\bq_k}\bigg\{\bq_k^TW_k\bq_k + 2\bp_k^TS_k^T\bq_k + 2\bl_k^TD_{k2}\bq_k + 2(A_k\bp_k+C_k\bl_k)^TK_{k+1}B_k\bq_k \nonumber\\
&- 2\sum_{i = k+1}^{N-1}\bl_i^TM_{i}^{k+1}B_k\bq_k-2\sum_{i=k+1}^{N-1}\bl_i^TC_i^TV_i^{k+1}B_k\bq_k\bigg\} \nonumber\\
= &\min_{\bq_k}\bigg\{\bq_k^TW_k\bq_k + 2 \ba_k^T\bq_k\bigg\} + \bp_k^T(Q_k+A_k^TK_{k+1}A_k)\bp_k + 2\bb_k^T\bp_k \label{pequ:4}\\
&+ \underbrace{T_{k+1} + \bl_k^TC_k^TK_{k+1}C_k\bl_k - 2\sum_{i=k+1}^{N-1}\bl_i^TM_i^{k+1}C_k\bl_k-2\sum_{i=k+1}^{N-1}\bl_i^TC_i^TV_i^{k+1}C_k\bl_k}_{T_k'} \nonumber,
\end{align}
where $\bb_k = (A_k^TK_{k+1}C_k+D_{k1}^T)\bl_k - \sum_{i=k+1}^{N-1}A_k^T\big((M_i^{k+1})^T+(V_i^{k+1})^TC_i\big)\bl_i$ and
\begin{multline*}
\ba_k = S_k\bp_k + D_{k2}^T\bl_k+B_k^TK_{k+1}(A_k\bp_k+C_k\bl_k)\\
-\sum_{i=k+1}^{N-1}B_k^T(M_i^{k+1})^T\bl_i-\sum_{i=k+1}^{N-1}B_k^T(V_i^{k+1})^TC_i\bl_i.
\end{multline*}
From the definition of $P_k$, the unique solution $\bq_k$ of (\ref{pequ:4}), at fixed $\bp_k$, is then given~by
\begin{align}\label{pequ:5}
\bq_k(\bp_k) =& -W_k^{-1}\ba_k \nonumber\\
=& P_k\bp_k + W_k^{-1}B_k^T\sum_{i = k+1}^{N-1}(M_i^{k+1})^T\bl_i + W_k^{-1}B_k^T\sum_{i = k+1}^{N-1}(V_i^{k+1})^TC_i\bl_i\\
&-W_k^{-1}(D_{k2}+C_k^TK_{k+1}B_k)^T\bl_k. \nonumber
\end{align}
This verifies the induction step for claim (ii). Plugging (\ref{pequ:5}) into equation (\ref{pequ:4}) and noting that $M_i^{k+1}E_k =  M_i^k$, $V_i^{k+1}E_k =  V_i^k$ and the definitions in (\ref{equ:lem4:1}), by some extended but straightforward calculations we get 
\begin{align}\label{pequ:6}
J_k(\bp_k) = &\bp_k^TK_{k}\bp_k + 2\bigg(\bl_k^T(D_{k1}+D_{k2}P_k)+\bl_k^TC_k^TK_{k+1}E_k -\sum_{i=k+1}^{N-1}\bl_i^TM_i^{k}\nonumber \\ & -\sum_{i=k+1}^{N-1}\bl_i^TC_i^TV_i^{k}\bigg)\bp_k+T_k\nonumber\\
= & \bp_k^TK_{k}\bp_k - 2 \sum_{i=k}^{N-1}\bl_i^TM_i^{k}\bp_k-2\sum_{i=k}^{N-1}\bl_i^TC_i^TV_i^{k}\bp_k+T_k.
\end{align}
Here, $T_k = T_k' - \big\|(D_{k2}^T+B_k^TK_{k+1}C_k)\bl_k-\sum_{i=k+1}^{N-1}B_k^T\big(M_i^{k+1} + C_i^TV_i^{k+1}\big)^T\bl_i\big\|_{W_k^{-1}}^2$. We see that  $T_k$  depends only on $\{\bl_i\}_{i=k}^{N-1}$; and if $\bl_i=0, \forall k\leq i\leq N-1$, we will have that $T_k = 0$. So we have proved claim (iii) and, thus, the statement. 
\end{proof}

Lemma \ref{lem:4} gives the explicit form for the optimal control variable $\bq$ for Problem~(\ref{equ:11}). This lemma will also be  useful in the next section when we analyze the optimal state variable. The next lemma shows how the convexification transformation relates to the cost-to-go matrices defined in the preceding lemma.

\begin{lemma}\label{lem:5}
When applying the convexification algorithm \ref{alg:1} with $\delta=0$, we will have $\tilde{R}_k(0) = W_k$, $\forall k\in[N-1]$, and $\bar{Q}_k(0) = K_k$, $\forall k\in[N]$.
\end{lemma}

\begin{proof}

We use reverse induction for $\bar{Q}_k(0)$, and $\tilde{R}_k(0)$ can be proved in concert. For the last stage $N$, we know that $\bar{Q}_N(0) = Q_N = K_N$. Suppose $\bar{Q}_{k+1}(0) = K_{k+1}$. Then using the definition of $W_k$ and $K_k$ in (\ref{equ:lem4:1}), we can get
\begin{align*}
\tilde{R}_k(0) =& R_k + B_k^T\bar{Q}_{k+1}(0)B_k = R_k + B_k^TK_{k+1}B_k = W_k,\\
\bar{Q}_k(0) = &\hat{Q}_k(0)-\tilde{Q}_k(0) = Q_k+A_k^T\bar{Q}_{k+1}(0)A_k - \tilde{Q}_k(0)\\
= & Q_k+A_k^T\bar{Q}_{k+1}(0)A_k - \tilde{S}_k^T(0)\tilde{R}_k^{-1}(0)\tilde{S}_k(0)\\
= & Q_k+A_k^TK_{k+1}A_k - (B_k^TK_{k+1}A_k+S_k)^TW_k^{-1}(B_k^TK_{k+1}A_k+S_k)\\
=&K_k,
\end{align*}
where the second equality from the end is due to the induction assumption and the last equality is due to (\ref{equ:lem4:1}). This completes the proof.
\end{proof}

\begin{corollary}\label{cor:1}
Under Assumption \ref{ass:1-2}, Algorithm \ref{alg:1} can be carried out successfully by setting $\delta = 0$. The output satisfies $\tilde{R}_k(0)\succ0$ and $\tilde{H}_k(0)\succeq0$.
\end{corollary}

\begin{proof}

Without linear terms, Lemma 10 in \cite{Verschueren2017Sparsity} proved a similar result (with a different approach). For self-consistency, we present a brief proof for our more general convexification procedure. By Lemma \ref{lem:4} (i) and Lemma \ref{lem:5}, we know that $\tilde{R}_k(0) = W_k\succ 0$ (and thus it is invertible) and
\begin{align*}
\tilde{H}_k(0) = &\begin{pmatrix}
\tilde{S}^T_k(0)\tilde{R}_k^{-1}(0)\tilde{S}_k(0) & \tilde{S}_k^T(0)\\
\tilde{S}_k(0) & \tilde{R}_k(0)
\end{pmatrix}\\
=&\begin{pmatrix}
I & \tilde{S}_k^T(0)\tilde{R}_k^{-1}(0)\\
0 & I
\end{pmatrix}\begin{pmatrix}
0 & 0\\
0 & \tilde{R}_k(0)
\end{pmatrix}\begin{pmatrix}
I & 0\\
\tilde{R}_k^{-1}(0)\tilde{S}_k(0) & I
\end{pmatrix}\succeq 0.
\end{align*}
\end{proof}

Based on these two lemmas, we know that when increasing $\delta$ slightly, by continuity of $\tilde{R}_k(\delta)$, we still can get $\tilde{R}_k(\delta)\succ 0$. In this case, we can show $\tilde{H}_k(\delta)\succ 0$ because the Schur complement of $\tilde{R}_k(\delta)$ is $\delta I\succ0$. However, since $\tilde{R}_k(\delta)$ is algebraically complicated, how large we can increase $\delta$ is not immediately clear. The following theorem presents a \textit{sufficient} interval for $\delta$ in the sense previously discussed.

\begin{theorem}[\textit{Sufficient} bound for $\delta$]\label{thm:4}
Under Assumption \ref{ass:2}, executing Algorithm \ref{alg:1} with $\delta\in(0,\gamma)$ will result in $\tilde{R}_k(\delta)\succ0$ and $\tilde{H}_k(\delta)\succ 0${\red, where $\gamma$ is defined in \eqref{equ:8}.}
\end{theorem}

\begin{proof}

For any $\delta\in(0,\gamma)$, we define $Q_k^\delta = Q_k - \delta I$ and $H_k^\delta = \begin{pmatrix}
Q_k^\delta & S_k^T\\
S_k & R_k
\end{pmatrix}$ and let $H^{\delta} = \text{diag}(H_0^{\delta}, \ldots, H_N^{\delta})$. Moreover, we define $I^{\delta} =( H- H^{\delta})/\delta = \diag(I, \0, I, \0, \ldots, I)$. Thus, we have $I - I^\delta = \diag(\0, I, \0, I, \ldots, \0)\succeq 0$ and $Z^TI^{\delta}Z = Z^TZ - Z^T(I - I^{\delta})Z \preceq I$, where the last inequality comes from the fact that $Z^TZ = I$ (note that $Z$ has orthonormal columns). Furthermore, from Assumption \ref{ass:2} we can get
\begin{align}\label{pequ:7}
Z^TH^\delta Z = Z^THZ-\delta Z^TI^{\delta}Z\succeq(\gamma-\delta)I\succ 0. 
\end{align}
Consider applying Algorithm \ref{alg:1} on $H^\delta$ with the same input matrices $\{A_k, B_k, C_k, D_k\}_{k=0}^{N-1}$ and shifting parameter $\tilde{\delta} = 0$. Its output is denoted as $\tilde{H}_k^\delta(0) = \begin{pmatrix}
\tilde{Q}_k^\delta(0) & \tilde{S}_k^{\delta T}(0)\\
\tilde{S}_k^{\delta}(0) & 	\tilde{R}_k^\delta(0)
\end{pmatrix}$. 
From (\ref{pequ:7}) we know that SOSC stated in Assumption \ref{ass:1-2} is satisfied for $H^{\delta}$. So, from Corollary \ref{cor:1}, we can get $\tilde{R}_k^\delta(0)\succ 0$. Furthermore, we have the following claims.\\
\textbf{Claim 1:} We have $\bar{Q}_k^\delta(0) = \bar{Q}_k(\delta)$, $\forall k\in[N]$, and $\tilde{R}_k^\delta(0) = \tilde{R}_k(\delta)$, $\forall k \in[N-1]$.\\
We prove the statement by reverse induction. For $k = N$, we know $\bar{Q}_N^\delta(0) = Q_N^\delta = Q_N - \delta I = \bar{Q}_N(\delta)$. Suppose $\bar{Q}_{k+1}^\delta(0) = \bar{Q}_{k+1}(\delta)$. Then we have
\begin{align*}
\tilde{R}_k^\delta&(0) = R_k + B_k^T\bar{Q}_{k+1}^\delta(0)B_k = R_k + B_k^T\bar{Q}_{k+1}(\delta)B_k = \tilde{R}_k(\delta),\\
\bar{Q}_{k}^\delta&(0) = \hat{Q}_k^\delta(0) - \tilde{Q}_k^\delta(0)\\
=&  Q_k^\delta + A_k^T\bar{Q}_{k+1}^\delta(0)A_k - (S_k+B_k^T\bar{Q}_{k+1}^\delta(0)A_k)^T\big(\tilde{R}_k^\delta(0)\big)^{-1}(S_k+B_k^T\bar{Q}_{k+1}^\delta(0)A_k)\\
= & Q_k + A_k^T\bar{Q}_{k+1}(\delta)A_k - (S_k+B_k^T\bar{Q}_{k+1}(\delta)A_k)^T\tilde{R}^{-1}_k(\delta)(S_k+B_k^T\bar{Q}_{k+1}(\delta)A_k)-\delta I\\
= & \bar{Q}_k(\delta),
\end{align*}
which proves Claim 1.\\
\textbf{Claim 2:} We have $\tilde{S}_k^\delta(0) = \tilde{S}_k(\delta)$, $\forall k\in[N-1]$, and $\tilde{Q}_k^\delta(0) = \tilde{Q}_k(\delta) - \delta I$, $\forall k\in[N]$.\\
From Claim 1, we know that $\forall k\in[N-1]$
\begin{align*}
\tilde{S}_k^\delta(0) =& S_k + B_k^T\bar{Q}_{k+1}^\delta(0)A_k = S_k + B_k^T\bar{Q}_{k+1}(\delta)A_k = \tilde{S}_k(\delta),\\
\tilde{Q}_k^{\delta}(0) =& \big(\tilde{S}_k^{\delta}(0)\big)^T\big(\tilde{R}_k^\delta(0)\big)^{-1}\tilde{S}_k^{\delta}(0) = \tilde{S}_k^T(\delta)\tilde{R}_k^{-1}(\delta)\tilde{S}_k(\delta) = \tilde{Q}_k(\delta)-\delta I.
\end{align*}
For $k = N$ we have $\tilde{Q}_N^\delta(0) = \0 = \tilde{Q}_N(\delta) - \delta I$. This completes the proof of Claim 2.\\
From these  two claims and Corollary \ref{cor:1}, we can get that  $\tilde{R}_k(\delta) = \tilde{R}_k^\delta(0)\succ 0$ and
\begin{align*}
\tilde{H}_k(\delta) = \begin{pmatrix}
\tilde{Q}_k(\delta) & \tilde{S}_k^T(\delta)\\
\tilde{S}_k(\delta) & \tilde{R}_k(\delta)
\end{pmatrix} =  \begin{pmatrix}
\tilde{S}_k^{\delta T}(0)\tilde{R}_k^\delta(0)^{-1}\tilde{S}_k^{\delta}(0) +\delta I & \tilde{S}_k^{\delta T}(0)\\
\tilde{S}_k^{\delta}(0) & 	\tilde{R}_k^\delta(0) 
\end{pmatrix}\succ 0.
\end{align*}
\end{proof}

\vskip -2pt
To summarize, in this section we propose a convexification procedure with linear shifting and derive a sufficient, explicit interval for the scalar $\delta$ that ensures that Algorithm \ref{alg:1} is well defined and results in positive definite matrices of the transformed quadratic program. In particular, we prove that under Assumption \ref{ass:2}, $\delta$ can be set as large as the minimum eigenvalue of the reduced Hessian matrix (i.e. $\gamma$). 

\vskip 1.5pt
In the next section, some detailed properties of the matrices appearing in the transformed QDP are discussed in preparation for  our main sensitivity analysis result.

\section{Uniform boundedness of convexified QDP}\label{sec:5}

We delve deeper into Algorithm \ref{alg:1} and analyze some properties of $\tilde{H}(\delta), \tilde{D}(\delta)$ that are critical for exponential decay results. Basically, we are going to analyze properties for the lower bound $\tilde{H}(\delta)$ and the upper bound $\tilde{D}(\delta)$ by some quantities independent of the time index that depend on both the parameter $\delta$ and the bounds on the  problem data at the unperturbed solution $(\bz^0, \blambda^0; \bd^0)$. Building on the boundedness results of this section, we further derive the decay rate of sensitivity to perturbations on $\bd^0$ in the next section. We start by presenting boundedness assumptions for long-horizon dynamic programming. Similar assumptions have been discussed in \cite{Xu2018Exponentially}.

\begin{definition}[Controllability matrix]\label{def:4}
For any starting stage $k\in[N-1]$ and evolution length $t>0$ such that $[k, k+t-1]\subset[N-1]$, we define the controllability matrix to be $\Xi_{k,t}=\big[B_{k+t-1}\; A_{k+t-1}B_{k+t-2}\; \ldots,\; \prod_{l=1}^{t-1}A_{k+l}B_k\big]$. 

\end{definition}

\begin{assumption}\label{ass:3}
We have following assumptions for Problem (\ref{equ:11}):\\
(i): There exists a constant $\Upsilon>0$ independent of $N$ such that $\|Q_k\|\vee \|R_k\|\vee \|S_k\|\vee \|D_{k1}\|\vee \|D_{k2}\|\vee \|A_k\|\vee \|B_k\|\vee \|C_k\|\leq \Upsilon$, for all $k$.\\
(ii): For any $ k\in[N-1]$, there exists $1\leq t_k\leq N-k$ such that $\Xi_{k,t_k}\Xi_{k,t_k}^T\succeq\lambda_CI$, and there exists $1\leq t\leq N$ such that $t_k\leq t$. Moreover, we assume that the positive parameters $\lambda_C$ and $t$ are independent from $k$ and $N$. 	
\end{assumption}

Similar to our uniform SOSC assumptions, the parameters $\Upsilon$, $\lambda_C$, and $t$ are assumed independent of horizon length $N$. This setup will allow us to have 
the sensitivity decay rate depend only on these parameters and not on $N$. Thus, as $N$ increases (i.e., we solve more and more stages), the sensitivity bounds remain unchanged. We note that Assumption \ref{ass:3} (i) is also assumed in \cite{Xu2018Exponentially} (see Assumption 2.1), while the statement in Assumption \ref{ass:3} (ii) is slightly different from its analogue (see Definition 2.2(b)), where the uniform controllability condition was borrowed from \cite{Keerthi1988Optimal} and considered a truncation of an infinite-horizon problem. In that case $\Xi_{k,t}$ is well defined even when $k+t> N$. To avoid defining the perturbation of infinite-horizon problem, we define our controllability matrix $\Xi_{k,t}$  only for $k+t\leq N$, since we do not assume to have access to $\{A_k, B_k\}_{k=N}^{\infty}$.

From Assumption \ref{ass:3}, we obtain that $\forall k\in[N-1]$ and $\forall 1\leq j\leq t_k$,
{\red
\begin{multline}\label{eq:boundControl}
\|\Xi_{k,j}\| = \|\big[B_{k+j-1} \ldots \prod_{l=1}^{j-1}A_{k+l}B_k \big]\|\\
\leq  \sum_{i = 0}^{j-1}\|\prod_{l = i+1}^{j-1}A_{k+l}B_{k + i}\|
\leq \sum_{i=0}^{j-1}\Upsilon^{j-i} \leq \frac{\Upsilon(1-\Upsilon^t)}{1-\Upsilon} \coloneqq \Psi.
\end{multline}
}
Now we are ready to present some properties for $\tilde{H}(\delta)$. Similar to the proof of Theorem~\ref{thm:4}, we  focus on $\tilde{H}(0)$ and then make use of the relationship between $\tilde{H}(\delta)$ and $\tilde{H}^{\delta}(0)$.

\begin{lemma}\label{lem:7}
Under Assumptions \ref{ass:1-2} and \ref{ass:3}, we have that $	\|\bar{Q}_k(0)\|\leq \Upsilon_{\bar{Q}}$, $\forall k\in[N]$ for some parameter $\Upsilon_{\bar{Q}}$ independent of $N$.
	
\end{lemma}

\begin{proof}

For any $k\in[N-1]$ and $\bar{\bp}_k\in \mR^{n_x}$, let us consider the following QDP without linear shift:
\vskip -18.3pt
\begin{subequations}\label{pequ:8}
\begin{align}
\min_{\substack{\bp_{k:N}\\ \bq_{k:N-1}}}\text{\ \ } & \sum_{i=k}^{N-1}\begin{pmatrix}
\bp_i \\ \bq_i
\end{pmatrix}^T\begin{pmatrix}
Q_i & S_i^T\\
S_i & R_i
\end{pmatrix}\begin{pmatrix}
\bp_i \\ \bq_i
\end{pmatrix}+\bp_{N}^TQ_N\bp_{N}, \label{pequ:8a}\\
\text{s.t.}\text{\ \ \ } & \bp_{i+1}=A_i\bp_i+B_i\bq_i, \text{\ \ \ \ \ \ } i = k, k+1, \ldots, N-1 \label{pequ:8b}\\
&\bp_k = \bar{\bp}_k. \label{pequ:8c}
\end{align}
\end{subequations}
Comparing Problem (\ref{pequ:8}) with Problem (\ref{equ:11}), we see that $D_k$ and $C_k$ vanish, but we know that (\ref{pequ:8}) still satisfies $Z^THZ\succ 0$.	From (\ref{equ:13}) in Lemma \ref{lem:4} with $\bl_{k:N-1} = \0$, we know that Problem (\ref{pequ:8}) has minimum value $\bar{\bp}_kK_k\bar{\bp}_k$, and by Lemma \ref{lem:5} we have the cost-to-go matrix $K_k = \bar{Q}_k(0)$. On the other hand, for $i\geq k$, successively applying $\bp_{i+1}=A_i\bp_i+B_i\bq_i$ returns that
\begin{align}\label{pequ:9}
\bp_{k+j}-\bigg(\prod_{l=0}^{j-1}A_{{k+l}}\bigg)\bar{\bp}_k=\Xi_{k,j}\begin{pmatrix}
\bq_{k+j-1}\\
\vdots\\
\bq_k
\end{pmatrix}\coloneqq \Xi_{k,j}\bq_{k:k+j-1}, \ \ \ \text{for}\ j\geq 0 .
\end{align}
We now consider a special feasible point defined as follows. Let 
\begin{equation}\label{eq:zeroControl}
\bar{\bq}_{k:k+t_k-1}=-\Xi_{k,t_k}^T(\Xi_{k,t_k}\Xi_{k,t_k}^T)^{-1}\big(\prod_{l=0}^{t_k-1}A_{k+l}\big)\bar{\bp}_k,
\end{equation}
and $\bar{\bp}_{k:k+t_k}$ is computed by using equation (\ref{pequ:9}) for $1\leq j\leq t_k$. Plugging \eqref{eq:zeroControl} into \eqref{pequ:9} with $j = t_k$, we can get $\bar{\bp}_{k+t_k} = 0$. Further, we let $\bar{\bq}_{k+t_k:N-1}=0$ and thus have $\bar{\bp}_{k+t_k+1:N}=\0$ from \eqref{pequ:9}. Our special feasible point can be summarized (in the reverse order) as 
\begin{subequations}\label{pequ:10}
\begin{align}
\bar{\bq}_{k:N-1} & \coloneqq\begin{pmatrix}
\bar{\bq}_{k+t_k:N-1}\\
\bar{\bq}_{k:k+t_k-1}\\
\end{pmatrix}=\begin{pmatrix}
\0\\
-\Xi_{k,t_k}^T(\Xi_{k,t_k}\Xi_{k,t_k}^T)^{-1}\bigg(\prod_{l=0}^{t_k-1}A_{k+l}\bigg)\bar{\bp}_k
\end{pmatrix}, \label{pequ:10a}\\
\bar{\bp}_{k:N} & \coloneqq\begin{pmatrix}
\bar{\bp}_{k+t_k:N}\\
\bar{\bp}_{k:k+t_k-1}\\
\end{pmatrix}=\begin{pmatrix}
\0\\
\bar{\bp}_{k:k+t_k-1}
\end{pmatrix}. \label{pequ:10b}
\end{align}
\end{subequations}
Combining equations (\ref{pequ:10}, (\ref{pequ:9}), and \eqref{eq:boundControl} and Assumption \ref{ass:3}, we can get the following:
\begin{align}\label{pequ:111}
\|\bar{\bq}_{k:N-1}\| \leq \frac{\Psi\Upsilon^t}{\lambda_C}\|\bar{\bp}_k\|, \text{\ \ } \|\bar{\bp}_{i}\|\leq (\Upsilon^{i-k}+\frac{\Psi^2\Upsilon^t}{\lambda_C})\|\bar{\bp}_k\|, \text{\ } \forall i\in[k+1, k+t_k-1].
\end{align}
We then have
\begin{align*}
\bar{\bp}_k^T\bar{Q}_k&(0)\bar{\bp}_k=\min_{\substack{(\bp_{k:N},  \bq_{k:N-1})\\ \text{satisfy (\ref{pequ:8b})}}} \sum_{i=k}^{N-1}\begin{pmatrix}
\bp_i \\ \bq_i
\end{pmatrix}^T\begin{pmatrix}
Q_i & S_i^T\\
S_i & R_i
\end{pmatrix}\begin{pmatrix}
\bp_i \\ \bq_i
\end{pmatrix}+\bp_{N}^TQ_{N}\bp_{N}\\
\leq& \sum_{i=k}^{N-1}\begin{pmatrix}
\bar{\bp}_i \\ \bar{\bq}_i
\end{pmatrix}^T\begin{pmatrix}
Q_i & S_i^T\\
S_i & R_i
\end{pmatrix}\begin{pmatrix}
\bar{\bp}_i \\ \bar{\bq}_i
\end{pmatrix}+\bar{\bp}_{N}Q_{N}\bar{\bp}_{N} =\sum_{i=k}^{k+t_k-1}\begin{pmatrix}
\bar{\bp}_i \\ \bar{\bq}_i
\end{pmatrix}^T\begin{pmatrix}
Q_i & S_i^T\\
S_i & R_i
\end{pmatrix}\begin{pmatrix}
\bar{\bp}_i \\ \bar{\bq}_i
\end{pmatrix}\\
\stackrel{(1)}{\leq}& 2\Upsilon\sum_{i=k}^{k+t_k-1}\|\bar{\bp}_i\|^2+2\Upsilon\sum_{i=k}^{k+t_k-1}\|\bar{\bq}_i\|^2\\
\stackrel{(2)}{\leq}&\underbrace{2\Upsilon\big(1+\frac{\Psi^2\Upsilon^{2t}}{\lambda_C^2}+\sum_{i=1}^{t-1}(\Upsilon^{i}+\frac{\Psi^2\Upsilon^t}{\lambda_C})^2\big)}_{\Upsilon_{\bar{Q}}}\|\bar{\bp}_k\|^2\coloneqq \Upsilon_{\bar{Q}}\|\bar{\bp}_k\|^2.
\end{align*}
Inequalities (1) and (2) hold by Assumptions \ref{ass:3} (i) and (\ref{pequ:111}), respectively. Since~$\bar{\bp}_k$ was chosen arbitrarily, we conclude the argument in the lemma by noting that $\Upsilon_{\bar{Q}}>\Upsilon\geq \|Q_N\| = \|\bar{Q}_N(0)\|$.
\end{proof}

Note that the upper bound would be difficult to obtain from Algorithm \ref{alg:1}, because of the challenge of bounding the terms $\tilde{S}_k^T(0)\tilde{R}_k^{-1}(0)\tilde{S}_k(0)$. We resolve the difficulty by making use of the controllability of the problem. A direct consequence of Lemma \ref{lem:7} is the boundedness of the cost-to-go matrix $K_k$. By Lemma~\ref{lem:5}, we have
\begin{align}\label{cor:2}
\|K_k\| = \|\bar{Q}_k(0)\|\leq \Upsilon_{\bar{Q}}.
\end{align}

Another consequence is the boundedness of $\tilde{H}(0)$.

\begin{corollary}\label{cor:3}

Under Assumptions \ref{ass:1-2} and \ref{ass:3}, the matrices $\tilde{Q}_k(0)$, $\tilde{R}_k(0)$, and $\tilde{S}_k(0)$ have a global upper bound independent of $N$.	
\end{corollary}

\begin{proof}
From Algorithm \ref{alg:1}, Assumption \ref{ass:3}, and Lemma \ref{lem:7}, we obtain the following:
\begin{align*}
\|\hat{Q}_k(0)\| =& \|Q_k+A_k^T\bar{Q}_{k+1}(0)A_k\|\leq \Upsilon+\Upsilon^2\Upsilon_{\bar{Q}},\\
\|\tilde{S}_k(0)\| =& \|S_k+B_k^T\bar{Q}_{k+1}(0)A_k\|\leq \Upsilon+\Upsilon^2\Upsilon_{\bar{Q}},\\
\|\tilde{R}_k(0)\| =& \|R_k+B_k^T\bar{Q}_{k+1}(0)B_k\|\leq \Upsilon+\Upsilon^2\Upsilon_{\bar{Q}}.	
\end{align*}
We also have  $\|\tilde{Q}_k(0)\|=\|\hat{Q}_k(0)-\bar{Q}_k(0)\|\leq \Upsilon+ \Upsilon_{\bar{Q}} + \Upsilon^2\Upsilon_{\bar{Q}}$.
\end{proof}

Based on the preceding results, we are ready to present the upper-bound-preserving property of Algorithm \ref{alg:1}.

\begin{theorem}[Upper-bound-preserving property]\label{thm:5}
For Problem (\ref{equ:11}), if Assumptions \ref{ass:2} and \ref{ass:3} hold, the output of Algorithm \ref{alg:1} for $\delta\in(0,\gamma)$ satisfies
\begin{align*}
\max\big\{\|\tilde{Q}_k(\delta)\|, \|\tilde{R}_k(\delta)\|, \|\tilde{S}_k(\delta)\|, \|\tilde{D}_{k1}(\delta)\|, \|\tilde{D}_{k2}(\delta)\|\big\}\leq \tUp, \ \ \forall k,
\end{align*}
for some constant $\tUp$ independent of $N$. Consequently, $\tilde{H}_k(\delta)$ is uniformly upper bounded in $k$.
	
\end{theorem}

\begin{proof}

For any $\delta\in(0,\gamma)$, similar to Theorem \ref{thm:4}, we define $Q_k^\delta = Q_k-\delta I$ and $H_k^\delta = \begin{pmatrix}
Q_k^\delta & S_k^T\\
S_k & R_k
\end{pmatrix}$. We focus only on $H^{\delta}$ for which Assumption \ref{ass:1-2} still holds. Applying Algorithm \ref{alg:1} on input matrices $\{H_k^{\delta}, A_k, B_k, C_k, D_k\}_{k=0}^{N-1}$ with shifting parameter $\tilde{\delta} = 0$, we denote the output by $\tilde{H}_k^\delta(0)$. Also note that Assumption \ref{ass:3} still holds for $H^\delta$. From Lemma \ref{lem:7} and Corollary \ref{cor:3} we thus have
\begin{align}\label{pequ:12}
\max\{\|\tilde{Q}_k^\delta(0)\|, \|\tilde{R}_k^\delta(0)\|, \|\tilde{S}_k^\delta(0)\|\}\leq \Upsilon_1,
\end{align}
for some constant $\Upsilon_1$. By Theorem \ref{thm:4} (see Claim 1 and Claim 2 in its proof), we know that
\begin{align}\label{pequ:13}
\tilde{Q}_k^\delta(0) = \tilde{Q}_k(\delta)-\delta I,\;	\tilde{R}_k^\delta(0) = \tilde{R}_k(\delta),\text{\ and \ }\tilde{S}_k^\delta(0) = \tilde{S}_k(\delta).
\end{align}
Combining (\ref{pequ:13}) with (\ref{pequ:12}), we obtain that $\tilde{Q}_k(\delta)$, $\tilde{R}_k(\delta)$, and  $\tilde{S}_k(\delta)$ are globally upper bounded by $\Upsilon_1+\delta$, implying $\tilde{H}_k(\delta)$ also has a global upper bound independent from $N$. Further, by Lemma \ref{lem:7} and Assumption \ref{ass:3} (i), we can get
\begin{align*}
\|\tilde{D}_{k1}(\delta)\| &= \|D_{k1} + C_k^T\bar{Q}_k(\delta)A_k\|\leq \Upsilon+\Upsilon^2\|\bar{Q}_k(\delta)\| \leq \Upsilon+\Upsilon^2(\Upsilon_1+\delta),\\
\|\tilde{D}_{k2}(\delta)\| &= \|D_{k2} + C_k^T\bar{Q}_k(\delta)B_k\|\leq \Upsilon+\Upsilon^2\|\bar{Q}_k(\delta)\| \leq \Upsilon+\Upsilon^2(\Upsilon_1+\delta).
\end{align*}
We have thus shown that  $\tilde{D}_{k1}(\delta)$ and $\tilde{D}_{k2}(\delta)$ have an upper bound uniform in $k$. We define $\tUp = \max(\Upsilon_1+\delta, \Upsilon+\Upsilon^2(\Upsilon_1+\delta))$, which completes the proof.
\end{proof}

As {\red a side note}, we see that if we let $\delta\in(0, \gamma\wedge 1)$, the upper bound $\tUp$ would be independent of $\delta$ and $\gamma$ by replacing $\delta$ by $1$ in the last terms of the above sequences of inequalities.

In the following, we will investigate the lower bound of the QDP obtained by convexification, Algorithm \ref{alg:1}. Here,  by the lower bound of a positive definite matrix, we mean its smallest eigenvalue. One can easily see that $\tilde{Q}_k(\delta) = \tilde{S}^T_k(\delta)\tilde{R}_k^{-1}(\delta)\tilde{S}_k(\delta)+\delta I\succeq \delta I$. On the other hand, there is no immediate guarantee that $\tilde{H}_k(\delta)$ has global lower bound away from 0 without some conditions on $\tilde{R}_k(\delta)$. For example, if $\tilde{R}_k(\delta)\succ 0$ but $\tilde{R}_k(\delta)\rightarrow 0$ as $k\rightarrow \infty$ and $\tilde{S}_k(\delta) = 0$, then $\tilde{H}_k(\delta) = \text{diag}(\delta I, \tilde{R}_k(\delta))\succ 0$ but does not have a lower bound away from 0 independently of $N$. Nevertheless, this adverse situation will not occur under our uniform SOSC (Assumption \ref{ass:2}). We will illustrate this fact through the next several theorems.

First, we present examples in the next two remarks to illustrate that (i) $\delta\in(0,\gamma)$ is sufficient but not necessary and (ii) only SOSC cannot guarantee a lower bound for $\tilde{H}_k(\delta)$ uniformly in $N$.

\begin{remark}\label{rem:1}
Suppose $S_k = 0$, $Q_k = R_k = \gamma I$. Then we know $H_k = \diag(Q_k, R_k) = \diag(\gamma I, \gamma I)$ and uniform SOSC holds with $Z^THZ\succeq \gamma I$ for any $\{A_k, B_k, C_k, D_k\}_{k=0}^{N-1}$. Let us apply Algorithm~\ref{alg:1} on $\{H_k\}_{k=0}^N$ with any matrices $\{A_k, B_k, C_k, D_k\}_{k=0}^{N-1}$ and set parameter $\delta = \gamma$. Then we have the output $\tilde{H}_k(\gamma) = \diag(\gamma I, \gamma I)\succeq \gamma I$. Thus, we see $\delta = \gamma$ is also feasible for this example although $\delta\in(0, \gamma)$ is crucial for us to prove Theorem \ref{thm:4}, because of (\ref{pequ:7}).

\end{remark}

\begin{remark}\label{rem:2}

Let $n_x=n_u=1$. For $k\in[N-1]$, suppose $A_k = B_k = C_k = S_k = 0$, $Q_k = 1$, and $R_k = 1/k$. We then have that $H_k = \text{diag}(Q_k, R_k) = \text{diag}(1, 1/k)\succ 0$, and it satisfies SOSC with $Z^THZ\succeq 1/k\cdot I$ (here $Z = I_{2N+1}$). By applying Algorithm~\ref{alg:1} with any $\delta>0$, we get $\tilde{H}_k(\delta) = \text{diag}(\delta, 1/k)$. However, we see $\lambda_{min}(\tilde{H}_k(\delta)) \leq 1/k\rightarrow 0$. So we cannot find a lower bound independent of $N$ under only SOSC.
\end{remark}

To make the proofs concise, we state some algebra preliminaries in the next~lemma. 

\begin{lemma}\label{lem:8} (i) For any matrix $A\in\mR^{m\times n}$, suppose $\tilde{A} = (\tilde{A}_{ij})\in \mR^{\tilde{m}\times\tilde{n}}$, $ \tilde{m} \geq m$, $\tilde{n} \geq n$ is its zero-padded extension (i.e., $\forall i\in [\tilde{m}]\backslash[m]$ or $ j\in[\tilde{n}]\backslash[n]$, $\tilde{A}_{ij} = 0$). Then we have $\|\tilde{A}\| = \|A\|$;\\
(ii) For a sequence of symmetric positive definite matrices $H_k = \begin{pmatrix}
A_k & B_k^T\\
B_k & C_k
\end{pmatrix}$, $k = 1, 2, \ldots$, where $A_k, C_k$ are square matrices, suppose $C_k\succeq \beta_CI\succ0$ and its Schur complement $A_k-B_k^TC_k^{-1}B_k\succeq\beta_S I\succ 0$, and $\|B_k\|\leq \beta_B$. Then we have $\forall k$
\begin{align*}
H_k\succeq\lambda_{BCS} I\succ 0, \text{\ \ \ }\text{where}\text{\ \ } \lambda_{BCS} = \bigg(\frac{\beta_C}{\beta_C + \beta_B}\bigg)^2\cdot\big(\beta_S\wedge \beta_C\big).
\end{align*}
\end{lemma}

\begin{proof}
(i). Without loss of generality, we assume that $\tilde{A} = \begin{pmatrix}
A & 0\\
0 & 0
\end{pmatrix}$; otherwise we just do permutations of $\tilde{A}$ that preserve the operator norm. Let $\tilde{A} = (B,\0)$, where $B^T = (A^T, \0)$. We then have 
\begin{align*}
\|\tilde{A}\|^2 = \lambda_{\max}(\tilde{A}\tilde{A}^T) = \lambda_{\max}(BB^T) = \|B\|^2 = \lambda_{\max}(B^TB) = \lambda_{\max}(A^TA) = \|A\|^2.
\end{align*}
\noindent (ii) Note that if $H_k$ is positive definite, $H_k\succeq\lambda_{BCS} I\Leftrightarrow \|H_k^{-1}\|\leq 1/\lambda_{BCS}$. Therefore, we  need only to find an upper bound on $H_k^{-1}$. We know that
\begin{align*}
H_k^{-1} =\begin{pmatrix}
I & 0\\
-C_k^{-1}B_k & I
\end{pmatrix}\begin{pmatrix}
(A_k-B_k^TC_k^{-1}B_k)^{-1} & 0\\
0 & C_k^{-1}
\end{pmatrix}\begin{pmatrix}
I & -B_k^TC_k^{-1}\\
0 & I
\end{pmatrix}\coloneqq \I_1\I_2\I_1^T.
\end{align*} 
Thus, we have $\|H_k^{-1}\|\leq \|\I_1\|^2\|\I_2\|$. By part (i), we know 
\begin{align}\label{pequ:14}
\|\I_1\| = \left\|\begin{pmatrix}
I & 0\\
0 & I
\end{pmatrix} + \begin{pmatrix}
0 & 0\\
-C_k^{-1}B_k & 0
\end{pmatrix} \right\|\leq 1 + \|C_k^{-1}B_k\|\leq 1+\beta_B/\beta_C.
\end{align}
For $\|\I_2\|$, we easily have that $\|\I_2\|\leq \max(\frac{1}{\beta_S}, \frac{1}{\beta_C})$. Using (\ref{pequ:14}), we obtain that
\begin{align*}
\|H_k^{-1}\|\leq (1+\frac{\beta_B}{\beta_C})^2\cdot\max(\frac{1}{\beta_S}, \frac{1}{\beta_C})\coloneqq \frac{1}{\lambda_{BCS}}.
\end{align*}
This completes the proof.
\end{proof}

From Theorem \ref{thm:5} we know that $\|\tilde{S}_k(\delta)\|\leq \tUp$, and the Schur complement of the bottom right matrix $\tilde{R}_k(\delta)$ in  $\tilde{H}_k(\delta)$ in Algorithm \ref{alg:1} is $\delta I$. Thus, to show $\tilde{H}_k(\delta)$ is lower bounded, we  need only a lower bound for $\tilde{R}_k(\delta)$, which we provide in the following lemma.

\begin{lemma}\label{lem:9}
Consider Problem (\ref{equ:11}). Under Assumption \ref{ass:2}, we execute Algorithm~\ref{alg:1} with $\delta\in(0,\gamma)$. Then we have that $\tilde{R}_k(\delta)\succeq\gamma I\succ0$, $\forall k\in[N-1]$.	
\end{lemma}

\begin{proof}

We use a technique similar to that in the proof of Lemma \ref{lem:7} and Theorem \ref{thm:4}. For any $k\in[N-1]$ and $\delta\in(0,\gamma)$, we define $Q_i^\delta = Q_i - \delta I, i\geq k$. For any $\bar{\bq}_k \in \mR^{n_u}$, consider the following QDP:
\begin{subequations}\label{pequ:18}
\begin{align}
\min_{\substack{\bp_{k+1:N}\\ \bq_{k+1:N-1}}} \text{\ \ } & \sum_{i=k+1}^{N-1}\begin{pmatrix}
\bp_i \\ \bq_i
\end{pmatrix}^T\begin{pmatrix}
Q_i^\delta & S_i^T\\
S_i & R_i
\end{pmatrix}\begin{pmatrix}
\bp_i \\ \bq_i
\end{pmatrix}+\bp_{N}^TQ_N^\delta\bp_{N}, \label{pequ:18a}\\
\text{s.t.} \text{\ \ \ } & \bp_{i+1}=A_i\bp_i+B_i\bq_i,\ \ \ \ \ \text{for }k+1 \leq i\leq N-1 \label{pequ:18b}\\
&\bp_{k+1} = B_k\bar{\bq}_k. \label{pequ:18c}
\end{align}
\end{subequations}
The objective function (\ref{pequ:18a}) can be rewritten as
\begin{multline}\label{pequ:16}
\sum_{i=k+1}^{N-1}\begin{pmatrix}
\bp_i \\ \bq_i
\end{pmatrix}^T\begin{pmatrix}
Q_i^\delta & S_i^T\\
S_i & R_i
\end{pmatrix}\begin{pmatrix}
\bp_i \\ \bq_i
\end{pmatrix}+\bp_{N}^TQ_N^\delta\bp_{N} \\
= \sum_{i=k+1}^{N-1}\begin{pmatrix}
\bp_i \\ \bq_i
\end{pmatrix}^T\begin{pmatrix}
Q_i & S_i^T\\
S_i & R_i
\end{pmatrix}\begin{pmatrix}
\bp_i \\ \bq_i
\end{pmatrix}+\bp_{N}^TQ_N\bp_{N} - \delta \|\bp_{k+1:N}\|^2.
\end{multline}
Let us define a special point $(\bar{\bp}, \bar{\bq})$ as follow: For $i\in[k-1]$, let $(\bar{\bp}_i, \bar{\bq}_i) = (\0,\0)$; let $(\bar{\bp}_k, \bar{\bq}_k) = (\0,\bar{\bq}_k)$; for $i>k$, let $\bar{\bp}_i$ be computed recursively by using the constraint $\bp_i = A_{i-1}\bp_{i-1}+B_{i-1}\bq_{i-1}$, and let $\bar{\bq}_i = \bq_i(\bar{\bp}_i)$, where $\bq_i(\cdot)$ is defined in (\ref{equ:12}) in Lemma \ref{lem:4} by replacing $\{Q_i\}_{i=0}^N$ by $\{Q_i^\delta\}_{i=0}^N$. Note that $\bq_i(\cdot)$ is well defined because $Z^TH^\delta Z\succ 0$, shown in (\ref{pequ:7}). By Assumption \ref{ass:2}, we have
\begin{align}\label{pequ:17}
\sum_{i=0}^{N-1}\begin{pmatrix}
\bar{\bp}_i \\ \bar{\bq}_i
\end{pmatrix}^T\begin{pmatrix}
Q_i & S_i^T\\
S_i & R_i
\end{pmatrix}&\begin{pmatrix}
\bar{\bp}_i \\ \bar{\bq}_i
\end{pmatrix}+\bar{\bp}_{N}^TQ_N\bar{\bp}_{N} \nonumber\\
= &\bar{\bq}^T_kR_k\bar{\bq}_k + \sum_{i = k+1}^{N-1}\begin{pmatrix}
\bar{\bp}_i \\ \bar{\bq}_i
\end{pmatrix}^T\begin{pmatrix}
Q_i & S_i^T\\
S_i & R_i
\end{pmatrix}\begin{pmatrix}
\bar{\bp}_i \\ \bar{\bq}_i
\end{pmatrix}+\bar{\bp}_{N}^TQ_N\bar{\bp}_{N} \nonumber\\
\geq &\gamma(\|\bar{\bp}\|^2+\|\bar{\bq}\|^2).
\end{align}
On the other hand, under our setup for the special point $(\bar{\bp},\bar{\bq})$ and from (\ref{equ:13}) in Lemma \ref{lem:4} (at stage $k+1$ with $\bl_i = \0, \; \forall i\geq k+1$), we know that $(\bar{\bp}_{k+1:N},\bar{\bq}_{k+1:N-1})$ has the objective value $\bar{\bp}_{k+1}^TK_{k+1}^{\delta}\bar{\bp}_{k+1}$ for Problem (\ref{pequ:18}), where $K_{k+1}^{\delta}$ denotes $K_{k+1}$ that is calculated by replacing $\{Q_i\}_{i= k+1}^N$ with $\{Q_i^\delta\}_{i= k+1}^N$. Further, from Lemma \ref{lem:5} and Claim 1 in Theorem \ref{thm:4}, we have that $K_{k+1}^{\delta} = \bar{Q}_{k+1}^\delta(0) = \bar{Q}_{k+1}(\delta)$. Therefore, combining (\ref{pequ:16}) and (\ref{pequ:17}), we  obtain the following: 
\begin{align*}
\gamma(\|\bar{\bp}\|^2+\|\bar{\bq}\|^2)&\stackrel{(\ref{pequ:17})}{\leq}     \bar{\bq}^T_kR_k\bar{\bq}_k + \sum_{i = k+1}^{N-1}\begin{pmatrix}
\bar{\bp}_i \\ \bar{\bq}_i
\end{pmatrix}^T\begin{pmatrix}
Q_i & S_i^T\\
S_i & R_i
\end{pmatrix}\begin{pmatrix}
\bar{\bp}_i \\ \bar{\bq}_i
\end{pmatrix}+\bar{\bp}_{N}^TQ_N\bar{\bp}_{N}\\
\stackrel{(\ref{pequ:16})}{=} & \bar{\bq}^T_kR_k\bar{\bq}_k + \delta\|\bar{\bp}_{k+1:N}\|^2+\sum_{i=k+1}^{N-1}\begin{pmatrix}
\bar{\bp}_i \\ \bar{\bq}_i
\end{pmatrix}^T\begin{pmatrix}
Q_i^\delta & S_i^T\\
S_i & R_i
\end{pmatrix}\begin{pmatrix}
\bar{\bp}_i \\ \bar{\bq}_i
\end{pmatrix}+\bar{\bp}_{N}^TQ_N^\delta\bar{\bp}_{N}\\
=&\bar{\bq}^T_kR_k\bar{\bq}_k + \delta\|\bar{\bp}_{k+1:N}\|^2 + \bar{\bp}_{k+1}^TK_{k+1}^{\delta}\bar{\bp}_{k+1}\\
=&\bar{\bq}^T_kR_k\bar{\bq}_k + \delta\|\bar{\bp}_{k+1:N}\|^2 + \bar{\bp}_{k+1}^T\bar{Q}_{k+1}(\delta)\bar{\bp}_{k+1}\\
= & \bar{\bq}^T_kR_k\bar{\bq}_k + \delta\|\bar{\bp}_{k+1:N}\|^2 + \bar{\bq}_{k}^TB_k^T\bar{Q}_{k+1}(\delta)B_k\bar{\bq}_{k}\\
= & \bar{\bq}^T_k\tilde{R}_k(\delta)\bar{\bq}_k + \delta\|\bar{\bp}_{k+1:N}\|^2\\
= &\bar{\bq}^T_k\tilde{R}_k(\delta)\bar{\bq}_k + \delta\|\bar{\bp}\|^2.
\end{align*}
Thus, if $\delta\in(0,\gamma)$, we have  
\begin{align*}
\bar{\bq}^T_k\tilde{R}_k(\delta)\bar{\bq}_k\geq \gamma\|\bar{\bq}\|^2 + (\gamma-\delta)\|\bar{\bp}\|^2\geq \gamma \|\bar{\bq}\|^2\geq \gamma \|\bar{\bq}_k\|^2.
\end{align*}
This concludes the proof.
\end{proof}

We summarize the lower-boundedness property of $\tilde{H}_k$ in the next theorem.

\begin{theorem}[Lower-bound-preserving property]\label{thm:6}
For Problem (\ref{equ:11}), if Assumptions \ref{ass:2} and \ref{ass:3} hold, we have, when executing Algorithm \ref{alg:1} with $\delta\in(0,\gamma)$, that $\forall k\in[N]$
\begin{align}\label{pequ:11}
\tilde{H}_k(\delta)\succeq \lambda_H I\succ 0,  \text{\ \ \ } \text{where} \text{\ \ } \lambda_H = \bigg(\frac{\gamma}{\gamma + \tilde{\Upsilon}}\bigg)^2\cdot\delta.
\end{align}
\end{theorem}

\begin{proof}

For any $k\in[N-1]$, we know 
\begin{align*}
\tilde{H}_k(\delta) = \begin{pmatrix}
\tilde{Q}_k(\delta) & \tilde{S}_k^T(\delta)\\
\tilde{S}_k(\delta) & \tilde{R}_k(\delta) 
\end{pmatrix} = \begin{pmatrix}
\tilde{S}_k^T(\delta)\big(\tilde{R}_k(\delta) \big)^{-1}\tilde{S}_k(\delta) + \delta I &  \tilde{S}_k^T(\delta)\\
\tilde{S}_k(\delta) & \tilde{R}_k(\delta) 
\end{pmatrix}.
\end{align*}
Combining Theorem \ref{thm:5}, Lemma \ref{lem:8} (ii), and Lemma \ref{lem:9} and using $\delta\in(0, \gamma)$, we~have
\begin{align*}
\lambda_H = \bigg(\frac{\gamma}{\gamma + \tilde{\Upsilon}}\bigg)^2\cdot\delta.
\end{align*}
Note that $\lambda_H<\delta$. Thus, this bound also holds for $k = N$. 
\end{proof}

From Theorems \ref{thm:5} and \ref{thm:6}, we obtained that Algorithm \ref{alg:1} has a uniform boundedness-preserving property: when the problem is uniform upper bounded (i.e., Assumption \ref{ass:3} holds), the convexified problem is also uniform upper bounded. Moreover, if Assumption \ref{ass:2} holds even if Hessian matrices $H_k$ are indefinite, not only are the output Hessian matrices $\{\tilde{H}_k(\delta)\}_{k\in[N]}$ positive definite, as shown in Theorem \ref{thm:4},  but they also have global lower bounds away from zero independent of $N$. This observation will have important consequences in sensitivity analysis in section \ref{sec:6}. We further note that $\lambda_H$, the lower bound of $\tilde{H}(\delta)$ in (\ref{pequ:11}), increases with both $\gamma$ and $\delta$. In section \ref{sec:6} we will see that the larger the lower bound, the faster the decay rate.

\section{Exponential decay of NLDP sensitivities}\label{sec:6}

In this section, we  focus on bounding the magnitude of $\bp_k$ and $\bq_k$, the optimal solution of Problem \eqref{equ:11}. From Lemma \ref{lem:3}, we know that applying Algorithm \ref{alg:1} to convexify it will not change the {\red primal} solution. For a concise notation, we assume that Algorithm \ref{alg:1} has been applied and the resulting QDP is still denoted as \eqref{equ:11}; moreover, the quantities defining it also satisfy the various bounds proved in section \ref{sec:4} and section \ref{sec:5}. Thus we will abuse the notation of $H_k$ and $D_k$ to truly indicate $\tilde{H}_k(\delta)$ and $\tilde{D}_k(\delta)$ for some $\delta\in(0,\gamma)$.

The main result of this paper is that the components of the sensitivity have exponential decay with respect to $|k-i|$ when setting $\bl\in\Pi_i$ where $\forall i\in\{-1\}\cup[N-1]$
\begin{equation}\label{set:pi}
\begin{aligned}
\Pi_{-1} = &\{\bl\in\mR^{n_x+Nn_d}: \|\bl\| = 1, \text{\ \ } \text{supp}(\bl)\subseteq[1, n_x]\},\\
\Pi_i = &\{\bl\in\mR^{n_x+Nn_d}: \|\bl\| = 1, \text{\ \ } \text{supp}(\bl)\subseteq[n_x+in_d+1, n_x+(i+1)n_d]\}.
\end{aligned}
\end{equation}
Recall the perturbation path in (\ref{equ:5}). We know $\bl = (\bl_{-1}; \bl_0; \ldots; \bl_{N-1})\in\Pi_i$ indicates that the perturbation  occurs only at stage $i$ and, in particular,  $i=-1$ indicates the perturbation at initial conditions. The key tool is the explicit solution to Problem~\eqref{equ:11} by combining (\ref{equ:12}) in Lemma \ref{lem:4} and the  recursive (dynamics) constraints. Based on its closed form, we further take advantage of the boundedness property analyzed in section~\ref{sec:5} to derive the decay rate. All matrices defined in Lemma \ref{lem:4} are computed from the QDP convexified by Algorithm \ref{alg:1}.

The following lemma provides the closed form of optimal $\bp$ of Problem (\ref{equ:11}). We  also mention that earlier \cite{Xu2018Exponentially} (see Proposition 2.6) Xu and Anitescu proposed similar results for the case that $D_k = 0$ and $C_k = I$. Our result is a straightforward extension.

\begin{lemma}[Closed form of $\bp$]\label{lem:10}
Define $O_k = B_kW_k^{-1}B_k^T$, $\forall k\in[N-1]$. Then the optimal solution $\bp$ of Problem \eqref{equ:11} at stage $k\in[N]$ is given (with notations from Lemma \ref{lem:4}) by
\begin{align*}
\bp_k = \bigg(\prod_{i=0}^{k-1}E_i\bigg)\bl_{-1} + \sum_{i = 0}^{N-1}U_i^k\bl_i + \sum_{i=0}^{N-1}F_i^kC_i\bl_i,
\end{align*}
where for $\forall i\in[N-1], k \in[N]$,
\begin{align*}
U_i^k =& \sum_{s=0}^{i\wedge k-1}\bigg(\prod_{l = s+1}^{k-1}E_l\bigg)O_s(M_i^{s+1})^T-\bigg(\prod_{l = i+1}^{k-1}E_l\bigg)B_iW_i^{-1}D_{i2}^T\pmb{1}_{i+1\leq k},\\
F_i^k =& \sum_{s=0}^{i\wedge k-1}\bigg(\prod_{l = s+1}^{k-1}E_l\bigg)O_s(V_i^{s+1})^T+\bigg(\prod_{l = i+1}^{k-1}E_l\bigg)(I-O_iK_{i+1})\pmb{1}_{i+1\leq k}.
\end{align*}
\end{lemma}

\begin{proof}
The argument holds for $k=0$ trivially. For $k=1$, using (\ref{equ:12}), we have
\begin{align*}
\bp_1 =& A_0\bp_0 + B_0\bq_0(\bp_0) + C_0\bl_0\\
=& A_0\bp_0 + B_0\bigg(P_0\bp_0 + W_0^{-1}B_0^T\sum_{i=1}^{N-1}(M_i^1)^T\bl_i + W_0^{-1}B_0^T\sum_{i=1}^{N-1}(V_i^1)^TC_i\bl_i\\
&-W_0^{-1}B_0^TK_1C_0\bl_0-W_0^{-1}D_{02}^T\bl_0\bigg) +C_0\bl_0\\
=& E_0\bl_{-1}+ O_0\sum_{i=1}^{N-1}(M_i^1)^T\bl_i + O_0\sum_{i=1}^{N-1}(V_i^1)^TC_i\bl_i -B_0W_0^{-1}D_{02}^T\bl_0+ (I-O_0K_1)C_0\bl_0,
\end{align*}
which is consistent with the argument at $k=1$. Generally, we have 
\begin{align}\label{pequ:100}
\bp_{k+1} =& A_k\bp_k + B_k\bq_k(\bp_k) + C_k\bl_k \nonumber\\
=& E_k\bp_k + B_k\bigg(W_k^{-1}B_k^T\sum_{i = k+1}^{N-1} (M_i^{k+1})^T\bl_i + W_k^{-1}B_k^T\sum_{i=k+1}^{N-1}(V_{i}^{k+1})^TC_i\bl_i \nonumber\\
&-W_k^{-1}B_k^TK_{k+1}C_k\bl_k-W_k^{-1}D_{k2}^T\bl_k\bigg)+C_k\bl_k \nonumber\\
=& E_k\bp_k + \underbrace{O_k\sum_{i = k+1}^{N-1} (M_i^{k+1})^T\bl_i}_{\alpha^1_k} + \underbrace{O_k\sum_{i=k+1}^{N-1}(V_{i}^{k+1})^TC_i\bl_i}_{\alpha^2_k} \underbrace{-B_kW_k^{-1}D_{k2}^T\bl_k}_{\alpha^3_k} \nonumber\\
&+\underbrace{(I-O_kK_{k+1})C_k\bl_k}_{\alpha^4_k} \nonumber\\
=& E_k\bp_k+\alpha^1_k+\alpha^2_k+\alpha^3_k+\alpha^4_k = E_kE_{k-1}\bp_{k-1} + \sum_{j=1}^4(\alpha^j_k+E_k\alpha^j_{k-1}) \nonumber\\
= & \cdots = (\prod_{i=0}^kE_i)\bp_0 + \sum_{j=1}^{4}\sum_{i = 0}^{k}(\prod_{l = i+1}^{k}E_l)\alpha^j_i \coloneqq (\prod_{i=0}^kE_i)\bp_0 + \sum_{j=1}^{4}\T_j.
\end{align}
For the terms $\T_3$ and $\T_4$, we have
\begin{equation}\label{pequ:20}
\begin{aligned}
\T_3 = &\sum_{i = 0}^{k}(\prod_{l = i+1}^{k}E_l)\alpha^3_i = -\sum_{i = 0}^{k}(\prod_{l = i+1}^{k}E_l)B_i^TW_i^{-1}D_{i2}^T\bl_i,\\
\T_4 = &\sum_{i = 0}^{k}(\prod_{l = i+1}^{k}E_l)\alpha^4_i = \sum_{i = 0}^{k}(\prod_{l = i+1}^{k}E_l)(I-O_iK_{i+1})C_i\bl_i.
\end{aligned}
\end{equation}
The terms $\T_1$ and $\T_2$, can be simplified analogously. We take $\T_1$ as an example. We have
\begin{align*}
\T_1 =&  \sum_{i = 0}^{k}(\prod_{l = i+1}^{k}E_l)\alpha^1_i = \alpha^1_k+E_k\alpha^1_{k-1}+...+E_kE_{k-1}...E_{1}\alpha^1_{0} \nonumber\\
= & O_k\sum_{i=k+1}^{N-1}(M_i^{k+1})^T\bl_i+E_kO_{k-1}\sum_{i=k}^{N-1}(M_i^{k})^T\bl_i+...+E_kE_{k-1}...E_{1}O_{0}\sum_{i=1}^{N-1}(M_i^{1})^T\bl_i \nonumber\\
= & E_kE_{k-1}...E_{1}O_{0}(M_{1}^{1})^T\bl_{1} +E_kE_{k-1}...E_{1}O_{0}(M_{2}^{1})^T\bl_{2}+E_kE_{k-1}...E_{2}O_{1}(M_{2}^{2})^T\bl_{2} \nonumber\\
&+E_kE_{k-1}...E_{1}O_{0}(M_{3}^{1})^T\bl_{3}+E_kE_{k-1}...E_{2}O_{1}(M_{3}^{2})^T\bl_{3}+E_kE_{k-1}...E_{3}O_{2}(M_{3}^{3})^T\bl_{3} \nonumber\\
&+\cdots \coloneqq \sum_{i=1}^{N-1}Z_i^{k+1}\bl_i,
\end{align*}
where $Z_i^{k+1}$ is defined as follows: $Z_i^{k+1}=\sum_{s=0}^{i-1}\big(\prod_{l=s+1}^{k}E_l\big)O_s(M_i^{s+1})^T$ for $1\leq i\leq k$; $Z_i^{k+1}=\sum_{s=0}^{k}\big(\prod_{l=s+1}^{k}E_l\big)O_s(M_i^{s+1})^T$ for $k+1\leq i\leq N-1$. So we get 
\begin{align}\label{pequ:24}
Z_i^{k+1}=\sum_{s=0}^{i\wedge(k+1)-1}\bigg(\prod_{l=s+1}^{k}E_l\bigg)O_s(M_i^{s+1})^T, \text{\ \ for\ } 1\leq i\leq N-1.
\end{align}
Similarly, we have $\T_2=\sum_{i=1}^{N-1}N_i^{k+1}C_i\bl_i$, where
\begin{align*}
N_i^{k+1}=\sum_{s=0}^{i\wedge (k+1)-1}\big(\prod_{l=s+1}^{k}E_l\big)O_s(V_i^{s+1})^T, \text{\ \ for\ } 1\leq i\leq N-1.
\end{align*}
From (\ref{pequ:20}) and (\ref{pequ:24}), we get
\begin{align}\label{pequ:27}
\T_1+\T_3 =& \sum_{i=1}^{N-1}Z_i^{k+1}\bl_i -\sum_{i = 0}^{k}(\prod_{l = i+1}^{k}E_l)B_i^TW_i^{-1}D_{i2}^T\bl_i \nonumber\\
=&\sum_{i=0}^{N-1}\bigg(Z_i^{k+1} - \big(\prod_{l = i+1}^{k}E_l\big)B_i^TW_i^{-1}D_{i2}^T\pmb{1}_{i\leq k}\bigg)\bl_i 
\coloneqq  \sum_{i=0}^{N-1}U_i^{k+1}\bl_i,
\end{align}
and
\begin{align}\label{pequ:28}
\T_2+\T_4 =& \sum_{i=1}^{N-1}N_i^{k+1}C_i\bl_i +\sum_{i = 0}^{k}(\prod_{l = i+1}^{k}E_l)(I-O_iK_{i+1})C_i\bl_i \nonumber\\
=&\sum_{i=0}^{N-1}\bigg(N_i^{k+1} + \big(\prod_{l = i+1}^{k}E_l\big)(I-O_iK_{i+1})\pmb{1}_{i\leq k}\bigg)C_i\bl_i \coloneqq  \sum_{i=0}^{N-1}F_i^{k+1}C_i\bl_i.
\end{align}
Taking (\ref{pequ:27}) and (\ref{pequ:28}) together and plugging back into (\ref{pequ:100}), we have
\begin{align*}
\bp_{k+1} =(\prod_{i=0}^kE_i)\bp_0 +  \sum_{i=0}^{N-1}U_i^{k+1}\bl_i+\sum_{i=0}^{N-1}F_i^{k+1}C_i\bl_i.
\end{align*}
We replace $k+1$ with $k$ and $\bp_0$ with $\bl_{-1}$ (initial conditions) and  conclude the proof.
\end{proof}

From this lemma, we see that $\bp_k$, $k\in[N]$, is a linear combination of $\{\bl_i\}_{i\in\{-1\}\cup[N-1]}$. To get the bound, we seek for $\bp_k$, we only need to analyze $U_i^k$ and $F_i^k$. One of their common terms is {\red the product of the closed-loop state transition matrix}, $\prod_{l = i}^{j}E_l$, which we now bound. {\red We first establish the relation of two consecutive cost-to-go matrices in next lemma.}

\begin{lemma}[{\red Riccati recursion}]\label{lem:11}
Following Definition (\ref{equ:lem4:1}) in Lemma \ref{lem:4}, we have $\forall k\in[N-1]$,
\begin{align*}
K_k = E_k^TK_{k+1}E_k + \begin{pmatrix}
I & P_k^T
\end{pmatrix}H_k\begin{pmatrix}
I\\
P_k
\end{pmatrix}.
\end{align*}
	
\end{lemma}

\begin{proof}

From (\ref{equ:lem4:1}), we have
\begin{align*}
K_k \stackrel{(\ref{equ:lem4:1})}{=} & Q_k + A_k^TK_{k+1}A_k-(B_k^TK_{k+1}A_k + S_k)^TW_k^{-1}(B_k^TK_{k+1}A_k + S_k)\\
= & Q_k + (E_k-B_kP_k)^TK_{k+1}(E_k-B_kP_k) + A_k^TK_{k+1}B_kP_k+S_k^TP_k\\
= & Q_k + E_k^TK_{k+1}E_k
-P_k^TB_k^TK_{k+1}E_k-E_k^TK_{k+1}B_kP_k\\
&+P_k^TB_k^TK_{k+1}B_kP_k + A_k^TK_{k+1}B_kP_k+S_k^TP_k\\
= & Q_k + E_k^TK_{k+1}E_k
-P_k^TB_k^TK_{k+1}E_k+S_k^TP_k\\
= &Q_k + E_k^TK_{k+1}E_k
-P_k^TB_k^TK_{k+1}A_k-P_k^TB_k^TK_{k+1}B_kP_k+S_k^TP_k\\
= &Q_k + E_k^TK_{k+1}E_k
-P_k^T(-W_kP_k-S_k)-P_k^TB_k^TK_{k+1}B_kP_k+S_k^TP_k\\
= &Q_k + E_k^TK_{k+1}E_k+P_k^TS_k + S_k^TP_k +
P_k^T(W_k-B_k^TK_{k+1}B_k)P_k\\
= &Q_k + E_k^TK_{k+1}E_k+P_k^TS_k + S_k^TP_k +
P_k^TR_kP_k\\
= & E_k^TK_{k+1}E_k+ \begin{pmatrix}
I & P_k^T
\end{pmatrix}H_k\begin{pmatrix}
I\\
P_k
\end{pmatrix}. 
\end{align*}
Here the second,  fifth, and sixth equalities are from the definition of $E_k$ and $P_k$ in Lemma~\ref{lem:4}, and the eighth equality is from the definition of $W_k$.
\end{proof}

From \eqref{cor:2}, we know that $\|K_k\|\leq \tilde{\Upsilon}_{\bar{Q}}$, $\forall k\in[N]$, for some constant $\tilde{\Upsilon}_{\bar{Q}}$ independent of $N$. Note that we distinguish $\tilde{\Upsilon}_{\bar{Q}}$ from ${\Upsilon}_{\bar{Q}}$ in (\ref{cor:2}) deliberately since our cost-to-go matrix $K_k$ in this section is calculated from the convexified problem, which still satisfies Assumptions \ref{ass:1-2} and \ref{ass:3}, as proved in Theorems \ref{thm:5} and \ref{thm:6}. Thus (\ref{cor:2}) holds for the convexified problem as well, although with a new bound. We use this result to bound the term $\prod_{l = i}^{j}E_l$.

\begin{lemma}\label{lem:12}
For any $0\leq i\leq j\leq N-1$, we have that $\|\prod_{l = i}^{j}E_l\|\leq \Upsilon_E\rho^{j-i+1}$ for some $\Upsilon_E$ and $\rho\in(0,1)$ independent from $N$.
\end{lemma}

\begin{proof}

For any $\bar{\bp}_i\in\mR^{n_x}$, we recursively define $\bar{\bp}_{l+1} = E_l\bar{\bp}_l$ for $i\leq l\leq j$. Based on Lemma \ref{lem:11}, we have that
\begin{align}\label{pequ:29}
\bar{\bp}_j^TK_j\bar{\bp}_j = &\bar{\bp}_j^TE_j^TK_{j+1}E_j\bar{\bp}_j + \bar{\bp}_j^T\begin{pmatrix}
I & P_k^T
\end{pmatrix}H_k\begin{pmatrix}
I\\
P_k
\end{pmatrix}\bar{\bp}_j \geq\bp_j^T\begin{pmatrix}
I & P_k^T
\end{pmatrix}H_k\begin{pmatrix}
I\\
P_k
\end{pmatrix}\bp_j \nonumber\\
\stackrel{(1)}{\geq} &\lambda_H\|\begin{pmatrix}
I\\
P_k
\end{pmatrix}\bp_j\|^2 \geq \lambda_H\|\bp_j\|^2,
\end{align}
where inequality (1) is due to Theorem \ref{thm:6} and $\lambda_H$ is defined in (\ref{pequ:11}). Further, by Theorems \ref{thm:5}, \ref{thm:6}, we know the transformed QDP satisfies Assumptions \ref{ass:1-2} and \ref{ass:3}. So we can apply Lemma \ref{lem:7}, and a consequently equation (\ref{cor:2}) holds for the transformed QDP. Therefore, we know that the cost-to-go matrices $K_k$ (for transformed QDP) satisfy $\|K_k\|\leq \tUp_{\bar{Q}}$ for some constant $\tUp_{\bar{Q}}$. In particular, from the comment after the proof of Theorem \ref{thm:5} we know that the uniform upper bounds of the transformed problem do not depend on the lower bound $\gamma$, and as a result, neither does $\tUp_{\bar{Q}}$. From this observation and Lemma \ref{lem:11}, we have 
\begin{multline}\label{pequ:26}
\bar{\bp}_j^TK_j\bar{\bp}_j = \bar{\bp}_{j+1}^TK_{j+1}\bar{\bp}_{j+1} + \bar{\bp}_j^T\begin{pmatrix}
I & P_k^T
\end{pmatrix}H_k\begin{pmatrix}
I\\
P_k
\end{pmatrix}\bar{\bp}_j\\
\geq  \bar{\bp}_{j+1}^TK_{j+1}\bar{\bp}_{j+1} + \lambda_H\|\bar{\bp}_j\|^2\geq  (1+\frac{\lambda_H}{\tUp_{\bar{Q}}})\bar{\bp}_{j+1}^TK_{j+1}\bar{\bp}_{j+1}.
\end{multline}
Note that in the last inequality we use the fact that $\bar{\bp}_{j}^TK_{j}\bar{\bp}_{j}\geq \bar{\bp}_{j+1}^TK_{j+1}\bar{\bp}_{j+1}$.  We obtain
\begin{align*}
\|\prod_{l=i}^{j}E_l\bar{\bp}_i\|^2=&\|\bar{\bp}_{j+1}\|^2\stackrel{(\ref{pequ:29})}{\leq}\frac{1}{\lambda_H}\bar{\bp}_{j+1}^TK_{j+1}\bar{\bp}_{j+1}\stackrel{(\ref{pequ:26})}{\leq} \frac{1}{\lambda_H(1+\lambda_H/\tUp_{\bar{Q}})}\bar{\bp}_j^TK_j\bar{\bp}_j\\
\leq &  \frac{1}{\lambda_H}\bigg(\frac{1}{1+\lambda_H/\tUp_{\bar{Q}}}\bigg)^{j-i+1}\bp_i^TK_i\bp_i\leq\frac{\tUp_{\bar{Q}}}{\lambda_H}\left(\frac{1}{1+\lambda_H/\tUp_{\bar{Q}}}\right)^{j-i+1}\|\bp_i\|^2.
\end{align*}
Letting $\Upsilon_E = \sqrt{\tUp_{\bar{Q}}/\lambda_H}$ and $\rho = 1/\sqrt{1+\lambda_H/\tUp_{\bar{Q}}}$, we conclude the proof.
\end{proof}

From  Lemma \ref{lem:4} we can immediately see  that $E_k$ is the closed loop matrix; thus the quantity in the preceding lemma is nothing but the relationship between $\bp_i$ and $\bp_j$ at optimality. From the proof, we have that $\rho$ can be
\begin{align}\label{equ:10}
\Upsilon_E = \frac{\sqrt{\tUp_{\bar{Q}}}}{\sqrt{\lambda_H}}, \text{\ \ \ } \rho= \frac{\sqrt{\tUp_{\bar{Q}}}}{\sqrt{\tUp_{\bar{Q}} + \lambda_H}}.
\end{align}
Here $\lambda_H$ is from Theorem \ref{thm:6} and $\tUp_{\bar{Q}}$ is the upper bound of cost-to-go matrices $K_k$ for the transformed QP. Both of them are independent of the horizon length $N$. According to the discussions after the proof of Theorem \ref{thm:6}, we see that when $\gamma$ increases, the lower bound $\lambda_H$ will also increase and $\Upsilon_E$ and $\rho$ will hence decrease, which means that the larger the lower bound $\gamma$, the faster the decay. We will see this argument can also be applied for $\bp_k$ and $\bq_k$.

\begin{remark}
Under SOSC, setting $\gamma = \lambda_{\min}(Z^THZ)>0$ works for finite horizons but will result in a rate depending on $N$. In practice, if we want to use the transformed QDP to solve the original Problem (\ref{equ:11}), we should make $\delta$ close to $\gamma$ as much as possible because the lower bound of transformed QDP, if $\gamma$ is fixed, is proportional to $\delta$ (see Theorem \ref{thm:6}). 

\end{remark}

Using Lemma \ref{lem:12}, we can show the decay rate for $U_i^k$ and $F_i^k$,  defined in Lemma~\ref{lem:10}.

\begin{lemma}\label{lem:13}
For any $i\in[N-1]$, $k\in[N]$, we have
\begin{align*}
\max(\|U_i^k\|, \|F_i^k\|)\leq \Upsilon_{uf}\rho^{|i-k|},
\end{align*}
for some constant $\Upsilon_{uf}$ independent of $N$ and $\rho\in(0,1)$ from Lemma \ref{lem:12}.
\end{lemma}

\begin{proof}

We first bound some terms appeared in $U_i^k$ and $F_i^k$. From Assumption~\ref{ass:3}, equation (\ref{cor:2}), Theorem \ref{thm:5}, and Lemma \ref{lem:9}, we have $\forall s\in[N-1]$
\begin{equation}\label{pequ:200}
\begin{aligned}
\|O_s\| =& \|B_sW_s^{-1}B_s^T\|\leq \Upsilon^2\|(R_s+B_s^TK_{s+1}B_s)^{-1}\|\leq \Upsilon^2\|R_s^{-1}\|\leq \Upsilon^2/\gamma,\\
\|P_s\| =& \|W_s^{-1}(B_s^TK_{s+1}A_s+S_s)\|\leq \frac{1}{\gamma}(\Upsilon^2\tUp_{\bar{Q}}+\tUp)\coloneqq \Upsilon_P.
\end{aligned}
\end{equation}
Here all the matrices are calculated from the transformed problem as defined in~Lemma \ref{lem:4}, as stated in the beginning of this section. As with (\ref{equ:10}), we use $\tUp_{\bar{Q}}$ to denote the upper bound for $K_k$, to distinguish from the constant in (\ref{cor:2}). A direct result based on the boundedness of $O_s$ and $P_s$ is an upper bound for $M_i^s$ and $V_i^s$. We have
\begin{align}\label{pequ:300}
\|M_i^s\|\leq (1+\Upsilon_P)\Upsilon\Upsilon_E\rho^{i-s},\ \ \ \ \|V_i^s\|\leq \tUp_{\bar{Q}}\Upsilon_E\rho^{i-s+1}\ \ \ \text{for\ } i\geq s.
\end{align}
Here $\Upsilon_E$ and $\rho$ come from Lemma \ref{lem:12}. So, by definition of $U_i^k$, $F_i^k$ in Lemma \ref{lem:10}, we can get
\begin{align}\label{pequ:30}
\|U_i^k\|&\leq \sum_{s=0}^{i\wedge k-1}\bigg(\Upsilon_E\rho^{k-s-1}\frac{\Upsilon^2}{\gamma}(1+\Upsilon_P)\Upsilon\Upsilon_E\rho^{i-s-1}\bigg)+\Upsilon_E\rho^{k-i-1}\frac{\Upsilon^2}{\gamma}\pmb{1}_{i+1\leq k} \nonumber\\
&=\frac{(1+\Upsilon_P)\Upsilon_E^2\Upsilon^3}{\gamma}\sum_{s=0}^{i\wedge k-1}\rho^{k+i-2s-2}+\frac{\Upsilon_E\Upsilon^2}{\gamma}
\rho^{k-i-1}\pmb{1}_{i+1\leq k} \nonumber\\
& = \frac{(1+\Upsilon_P)\Upsilon_E^2\Upsilon^3}{\gamma}\rho^{k+i-2}\frac{1-(1/\rho^2)^{i\wedge k}}{1-1/\rho^2}+\frac{\Upsilon_E\Upsilon^2}{\gamma}
\rho^{k-i-1}\pmb{1}_{i+1\leq k}.
\end{align}
We separate equation (\ref{pequ:30}) in different cases. If $i+1\leq k$, we know
\begin{align}\label{pequ:31}
\|U_i^k\|\leq & \frac{(1+\Upsilon_P)\Upsilon_E^2\Upsilon^3}{\gamma}\rho^{k+i-2}\frac{1-(1/\rho^2)^{i}}{1-1/\rho^2}+\frac{\Upsilon_E\Upsilon^2}{\gamma}
\rho^{k-i-1} \nonumber\\
= & \frac{(1+\Upsilon_P)\Upsilon_E^2\Upsilon^3}{\gamma}\cdot\frac{\rho^{k-i}-\rho^{k+i}}{1-\rho^2}+\frac{\Upsilon_E\Upsilon^2}{\gamma\rho}
\rho^{k-i} \nonumber\\
\leq & \underbrace{\bigg(\frac{(1+\Upsilon_P)\Upsilon_E^2\Upsilon^3}{\gamma(1-\rho^2)}+\frac{\Upsilon_E\Upsilon^2}{\gamma\rho}\bigg)}_{\Upsilon_u}\rho^{k-i}.
\end{align}	
If $k\leq i$, we know
\begin{align}\label{pequ:32}
\|U_i^k\|\leq & \frac{(1+\Upsilon_P)\Upsilon_E^2\Upsilon^3}{\gamma}\rho^{k+i-2}\frac{1-(1/\rho^2)^{k}}{1-1/\rho^2}\leq \Upsilon_u \rho^{i-k}.
\end{align}
Combining equation (\ref{pequ:31}) with equation (\ref{pequ:32}), we have
\begin{align}\label{pequ:33}
\|U_i^k\|\leq \Upsilon_u \rho^{|i-k|}, \forall i\in[N-1], k\in[N].
\end{align}
Analogously, for $\|F_i^k\|$, we have
\begin{align*}
\|F_i^k\|&\leq \sum_{s=0}^{i\wedge k-1}\bigg(\Upsilon_E\rho^{k-s-1}\frac{\Upsilon^2}{\gamma}\tUp_{\bar{Q}}\Upsilon_E\rho^{i-s}\bigg)+\Upsilon_E\rho^{k-i-1}(1+\frac{\Upsilon^2}{\gamma}\tUp_{\bar{Q}})\pmb{1}_{i+1\leq k} \nonumber\\
&=\frac{\Upsilon^2\Upsilon_E^2\tUp_{\bar{Q}}}{\gamma}\rho^{k+i-1}\sum_{s=0}^{i\wedge k-1}(\frac{1}{\rho^2})^s+(\Upsilon_E+\frac{\Upsilon^2\Upsilon_E\tUp_{\bar{Q}}}{\gamma})\rho^{k-i-1}\pmb{1}_{i+1\leq k} \nonumber\\
& = \frac{\Upsilon^2\Upsilon_E^2\tUp_{\bar{Q}}}{\gamma}\rho^{k+i-1}\frac{1-(1/\rho^2)^{i\wedge k}}{1-1/\rho^2}+(\Upsilon_E+\frac{\Upsilon^2\Upsilon_E\tUp_{\bar{Q}}}{\gamma})\rho^{k-i-1}\pmb{1}_{i+1\leq k}.
\end{align*}
If $i+1\leq k$, we get
\begin{align}\label{pequ:34}
\|F_i^k\|\leq& \frac{\Upsilon^2\Upsilon_E^2\tUp_{\bar{Q}}}{\gamma}\rho^{k+i-1}\frac{1-(1/\rho^2)^{i}}{1-1/\rho^2}+(\Upsilon_E+\frac{\Upsilon^2\Upsilon_E\tUp_{\bar{Q}}}{\gamma})\rho^{k-i-1} \nonumber\\
\leq & \underbrace{\bigg(\frac{\Upsilon^2\Upsilon_E^2\tUp_{\bar{Q}}\rho}{\gamma(1-\rho^2)}+\frac{\Upsilon_E}{\rho}+ \frac{\Upsilon^2\Upsilon_E\tUp_{\bar{Q}}}{\gamma\rho}\bigg)}_{\Upsilon_f}\rho^{k-i}.
\end{align}
If $k\leq i$, we know
\begin{align}\label{pequ:35}
\|F_i^k\|\leq& \frac{\Upsilon^2\Upsilon_E^2\tUp_{\bar{Q}}}{\gamma}\rho^{k+i-1}\frac{1-(1/\rho^2)^{k}}{1-1/\rho^2}= \frac{\Upsilon^2\Upsilon_E^2\tUp_{\bar{Q}}\rho}{\gamma(1-\rho^2)}(\rho^{i-k}-\rho^{k+i})\leq \Upsilon_f\rho^{i-k}.
\end{align}
Combining (\ref{pequ:34}) and (\ref{pequ:35}), we have	
\begin{align}\label{pequ:36}
\|F_i^k\|\leq \Upsilon_f \rho^{|i-k|}, \forall i\in[N-1], k\in[N].
\end{align}
Thus, combining (\ref{pequ:33})  and (\ref{pequ:36}) and letting $\Upsilon_{uf} = \max(\Upsilon_u, \Upsilon_f)$, we complete the~proof.
\end{proof}

The next remark discusses the dependence of constant $\Upsilon_{uf}$ and $\rho$ on $\gamma$.

\begin{remark}\label{rem:3}

Let us use ``$\nearrow$" (``$\searrow$") to denote the increase (decrease) of parameters. Combining the expression in (\ref{pequ:11}), (\ref{equ:10}), (\ref{pequ:31}) and (\ref{pequ:34}), we have that $\frac{1}{\gamma(1-\rho^2)}, \frac{\Upsilon_E}{\rho}, \frac{1}{\gamma\rho}$ all decrease when $\gamma$ increases. Therefore, we know
\begin{align*}
\gamma \nearrow \Longrightarrow \lambda_H \nearrow \Longrightarrow \rho, \Upsilon_E, \Upsilon_u, \Upsilon_f, \Upsilon_P \searrow \Longrightarrow \Upsilon_{uf} \searrow.
\end{align*}
All other upper bounds such as $\tUp_{\bar{Q}}, \Upsilon, \tilde{\Upsilon}$ are independent from $\gamma$. 	
\end{remark}

Based on the previous analysis, we can show the exponential decay for directional derivatives $(\bp_k,\bq_k)$. We use  their explicit form in Lemma \ref{lem:4} and Lemma~\ref{lem:10} combining with the upper bound in Lemma \ref{lem:12} and Lemma \ref{lem:13}. We summarize our main theoretical result in next theorem.

\begin{theorem}[Exponential decay for directional derivatives]\label{thm:7}
We consider the NLDP with only equality constraints as shown in Problem (\ref{pro:1}). Suppose we perturb the reference variable $\bd$ along direction $\bl$ as shown in (\ref{equ:5}), under Assumptions \ref{ass:2} and \ref{ass:3}. Then there exist parameters $\Upsilon_{pq}, \rho\in(0,1)$ independent from $N$ such that~$\forall k$
\begin{align*}
\max(\|\bp_k\|_2, \|\bq_k\|_2)\leq \Upsilon_{pq}\rho^k, \text{\ \ \ } \text{if} \text{\ \ } \bl\in \Pi_{-1},
\end{align*}
and
\begin{align*}
\max(\|\bp_k\|_2, \|\bq_k\|_2)\leq \Upsilon_{pq}\rho^{|k-i|},\text{\ \ \ } \text{if} \text{\ \ } \bl\in \Pi_i,\; \forall i\in[N-1],
\end{align*} 
where set $\Pi_i$ is defined in (\ref{set:pi}). Moreover, as $\gamma$ increases, $\Upsilon_{pq}$ and $\rho$ will decrease. Here $\bp_k$, $\bq_k$ are directional derivatives of $\bx_k$, $\bu_k$, respectively, defined as in (\ref{equ:6}).
\end{theorem}

\begin{proof}

When $\bl\in\Pi_{-1}$, by Lemma \ref{lem:10} and (\ref{equ:12}) in Lemma \ref{lem:4}, we have
\begin{align*}
\bp_k = \bigg(\prod_{i=0}^{k-1}E_i\bigg)\bl_{-1}, \text{\ \ \ } \bq_k = P_k\bp_k.
\end{align*}
Using Lemma \ref{lem:12} and (\ref{pequ:200}) in Lemma \ref{lem:13}, we have
\begin{align*}
\|\bp_k\|\leq \Upsilon_E\rho^k, \text{\ \ \ } \|\bq_k\|\leq \Upsilon_P\Upsilon_E\rho^k.
\end{align*}
Without loss of generality, we assume $\Upsilon_P>1$. So, we define $\Upsilon_{pq, 1} \coloneqq \Upsilon_P\Upsilon_E$ and finish the first part of proof. When $\bl\in\Pi_i$, $\forall i\in[N-1]$, by Lemma \ref{lem:10}, we have
\begin{align*}
\bp_k = & U_i^k\bl_i+F_i^kC_i\bl_i,\ \ \ \forall k\in[N].
\end{align*}
Then, from Lemma \ref{lem:13} and Assumption \ref{ass:3}, we can get
\begin{align}\label{pequ:37}
\|\bp_k\|  = & \|U_i^k\bl_i+F_i^kC_i\bl_i\|\leq \|U_i^k\| + \Upsilon\|F_i^k\| \leq  (1+\Upsilon)\Upsilon_{uf}\rho^{|k-i|}\coloneqq \Upsilon_{p}\rho^{|k-i|}.
\end{align}
Moreover, we get the upper bound $\bq_k$ using (\ref{equ:12}) in Lemma \ref{lem:4}. If $i\leq k-1$, we have
\begin{align}\label{pequ:38}
\|\bq_k\| = \|P_k\bp_k\|\leq \Upsilon_{p}\Upsilon_P\rho^{k-i}.
\end{align}
If $i = k$, we have
\begin{align}\label{pequ:39}
\|\bq_k\| =& \|P_k\bp_k-W_k^{-1}(B_k^TK_{k+1}C_k + D_{k2}^T)\bl_k\|
\leq \Upsilon_{p}\Upsilon_P+ \frac{\Upsilon^2\tUp_{\bar{Q}}+\Upsilon}{\gamma}.
\end{align}
where the inequality is due to the triangle inequality and $\|W_k^{-1}\|\leq \frac{1}{\gamma}$. If $i\geq k+1$,~then
\begin{align}\label{pequ:40}
\|\bq_k\| = & \|P_k\bp_k + W_k^{-1}B_k^{T}(M_i^{k+1})^T\bl_i+W_k^{-1}B_k^T(V_i^{k+1})^TC_i\bl_i\| \nonumber\\
\stackrel{(\ref{pequ:300})}{\leq} & \Upsilon_{p}\Upsilon_P\rho^{i-k} + \frac{\Upsilon}{\gamma}\bigg((1+\Upsilon_P)\Upsilon\Upsilon_E\rho^{i-k-1} + \Upsilon\tUp_{\bar{Q}}\Upsilon_E\rho^{i-k}\bigg) \nonumber\\
= &\bigg(\Upsilon_{p}\Upsilon_P + \frac{(1+\Upsilon_P+\rho \tUp_{\bar{Q}})\Upsilon_E\Upsilon^2}{\gamma\rho}\bigg)\rho^{i-k}
\end{align}
Combining (\ref{pequ:37}), (\ref{pequ:38}), (\ref{pequ:39}), (\ref{pequ:40}) and defining
\begin{align}\label{pequ:23}
\Upsilon_{pq, 2} = \Upsilon_{p}\Upsilon_P + \frac{(1+\Upsilon_P+\rho \tUp_{\bar{Q}})\Upsilon_E\Upsilon^2}{\gamma\rho}+\frac{\Upsilon^2\tUp_{\bar{Q}}+\Upsilon}{\gamma} > \Upsilon_{p},
\end{align}
 we get
\begin{align*}
\max(\|\bp_k\|, \|\bq_k\|)\leq \Upsilon_{pq, 2}\rho^{|i-k|}.
\end{align*}
This finishes the proof of the second part. To summarize, we complete the proof by defining $\Upsilon_{pq} = \max(\Upsilon_{pq, 1}, \Upsilon_{pq, 2})$. From Remark \ref{rem:3} and the definition of $\Upsilon_{pq, 1}$ and $\Upsilon_{pq, 2}$, we can  see that $\Upsilon_{pq}$ and $\rho$ will decrease when $\gamma$ increases.
\end{proof}

From the main theorem above, we see that the decay rate $\rho$ is the same as the one in Lemma \ref{lem:12}. From (\ref{pequ:11}) and (\ref{equ:10}), we know
\begin{align*}
\rho = \frac{\sqrt{\tUp_{\bar{Q}}}}{\sqrt{\tUp_{\bar{Q}}+\big(\frac{\gamma}{\gamma+\tUp}\big)^2\cdot\delta}}.
\end{align*}
To clarify, although the sensitivity decay rate of an NLDP should not depend on $\delta$, which is only a hyperparameter of the convexification algorithm, the above rate is our provable rate, and $\delta$ is treated as a parameter. Throughout the proof, we rely on setting $\delta\in(0,\gamma)$ to get uniform boundedness of the matrices. From a practical perspective, the suggested value is $\delta\approx \gamma$.

\section{Numerical experiment}\label{sec:7}

We conducted a numerical experiment to verify our theoretical result in Theorem \ref{thm:7}. For simplicity, we  chose a moderate temporal length and set $n_x = n_u = 1$. In particular, we studied the following nonconvex QDP:
\begin{subequations}\label{equ:14}
\begin{align}
\min_{\bx, \bu}\text{\ \ } & \sum_{k=0}^{N-1}\mu_1 (u_k-d_k)^2 - \mu_2 (x_k-d_k)^2 - \mu_2x_N^2, \label{equ:14a}\\
\text{ s.t.}\text{\ \ \ } & x_{k+1} = u_k+f(d_k),\ \ \forall k\in[N-1], \label{equ:14b}\\
&x_0 = 0. \label{equ:14c}
\end{align}
\end{subequations}
Here we assume that $f$ is any function such that $f(0)=0$. Note that we still have  coupled KKT conditions for our toy example. Suppose the unperturbed reference variable is $\bd^ 0 = (d_0,..., d_{N-1})=\0$. Then the unperturbed problem has a unique minimizer that is $x_k = u_k =0$, $\forall k$, provided $\mu_1>\mu_2>0$. Furthermore,  we can check that Assumption \ref{ass:2} holds with $\gamma =\mu_1 - \mu_2$  and that Assumption \ref{ass:3} (i) holds with $\Upsilon = \max(1, \mu_1, \mu_2, f'(0))$, and therefore the controllability condition in Assumption \ref{ass:3} (ii) holds with $t_k = t = \lambda_C = 1$.

\noindent{\bf Simulation setting:} We set $N = 40, 60$, $\mu = (\mu_1, \mu_2) = (10, 1), (50, 10), (100, 15)$, and we let $f(x) = x$ or $\exp(x)-1$. We  perturb only $\bd$ at the middle stage but with three different scales: $\epsilon = 1, 0.1, 0.01$. Note that only if $\epsilon\rightarrow 0$ can we approximate the directional derivatives. Because the unperturbed solution is all zeros, we plot $\log (|x_k^*(\epsilon)|/\epsilon)$ v.s. $k$, for which we expect to observe a linear decay. All optimization problems are solved by using the JuMP package \cite{Dunning2017JuMP} in Julia 0.6.4.1.

Simulation results are shown in Figure \ref{fig:1}. We see that the results all have an obvious linear decay trend, which is consistent with Theorem \ref{thm:7}.

\begin{figure}[!htp]
\centering
%
	\subfloat[$N=40$, $\mu = (10,1)$]{\label{w4}\includegraphics[width=4.7cm,height=4.8cm]{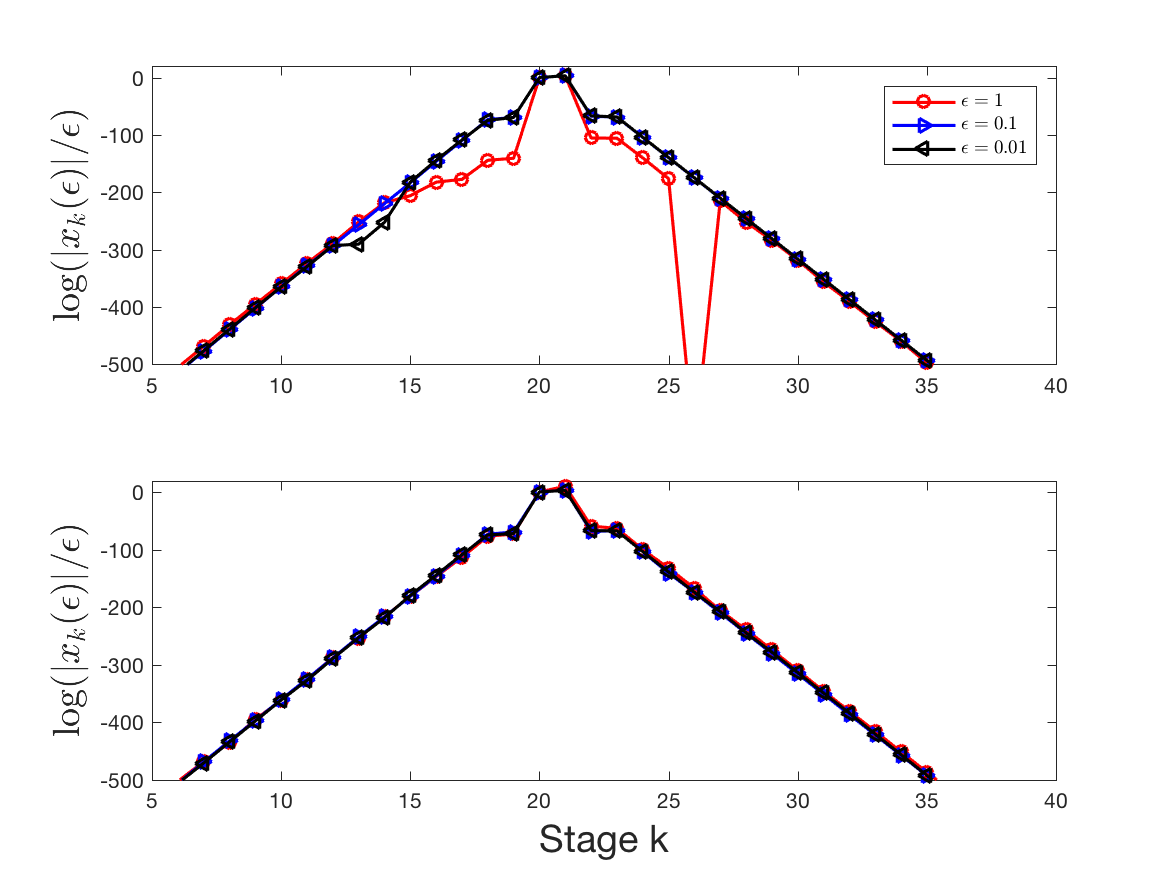}}
	\subfloat[$N=40$, $\mu = (50,10)$]{\label{w5}\includegraphics[width=4.7cm,height=4.8cm]{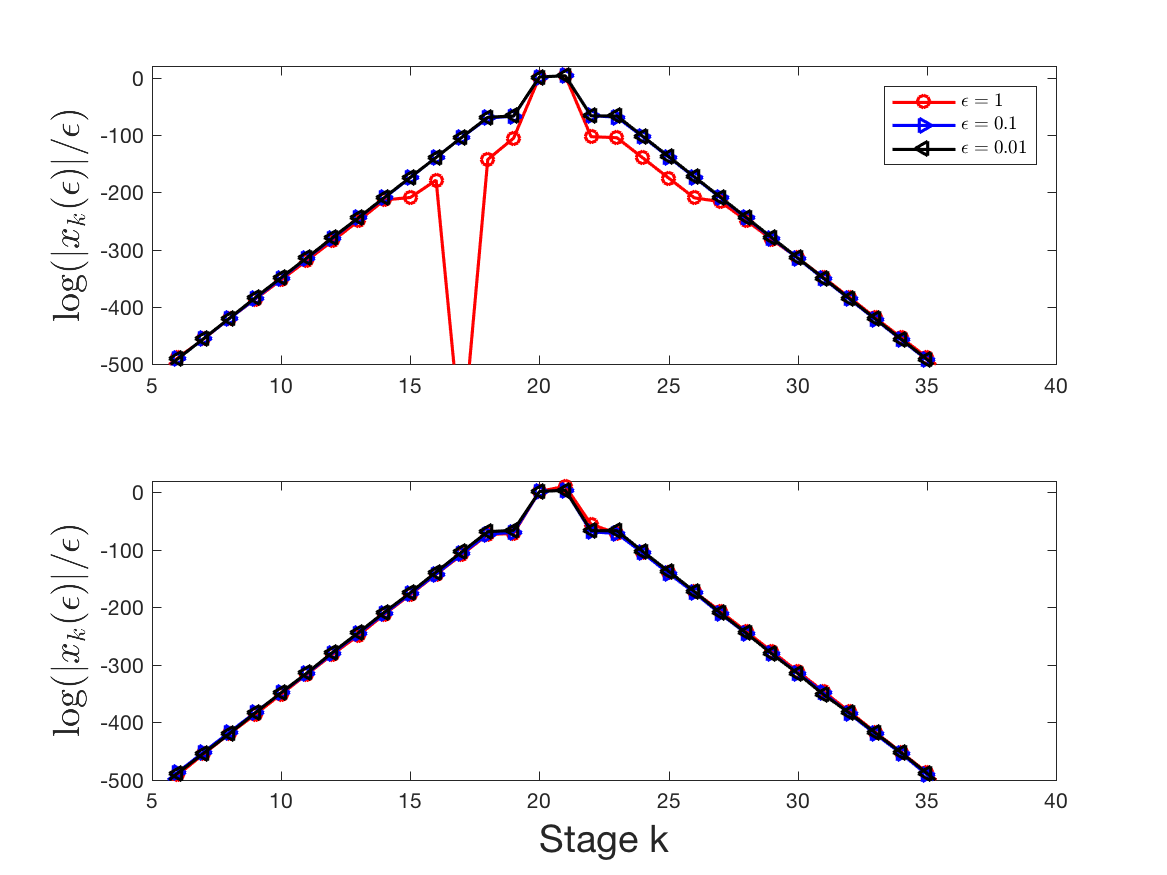}}
	\subfloat[$N=40$, $\mu = (100,15)$]{\label{w6}\includegraphics[width=4.7cm,height=4.8cm]{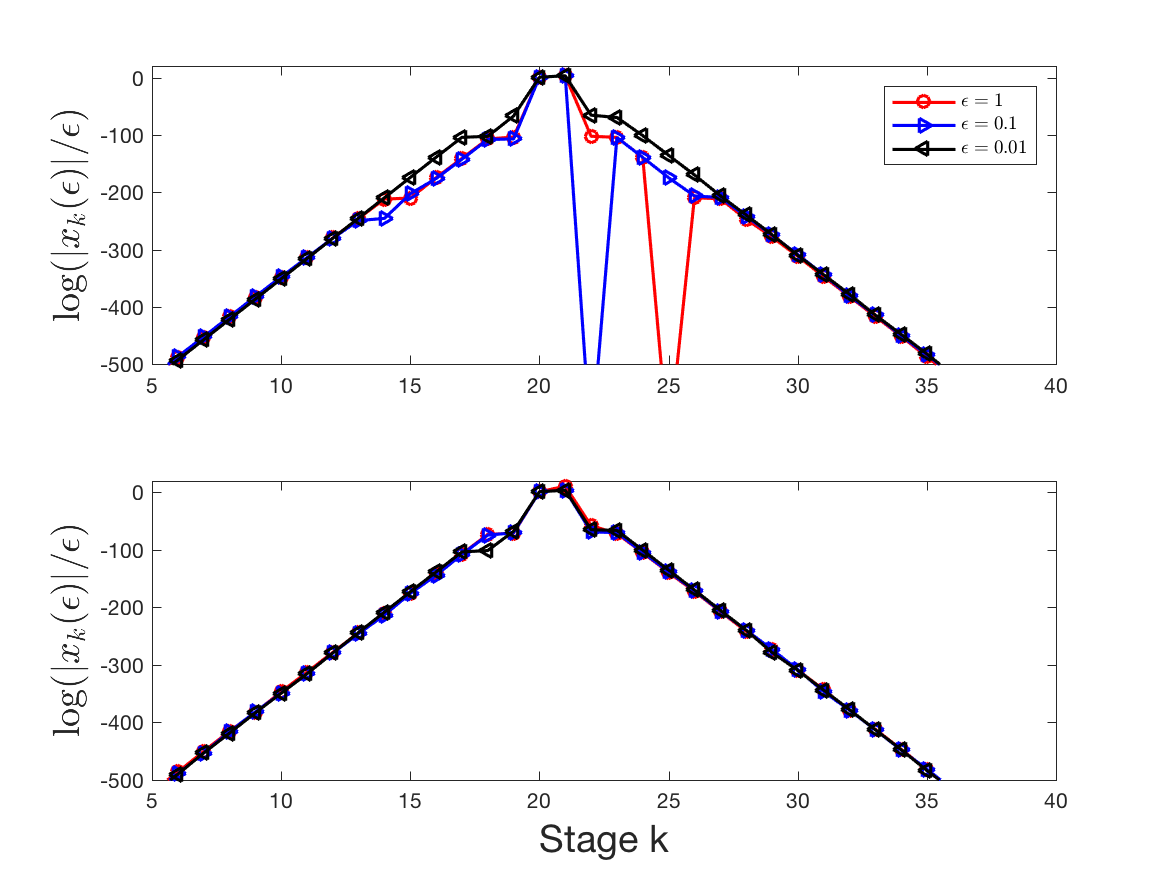}}
	
	\subfloat[$N=60$, $\mu = (10,1)$]{\label{w7}\includegraphics[width=4.7cm,height=4.8cm]{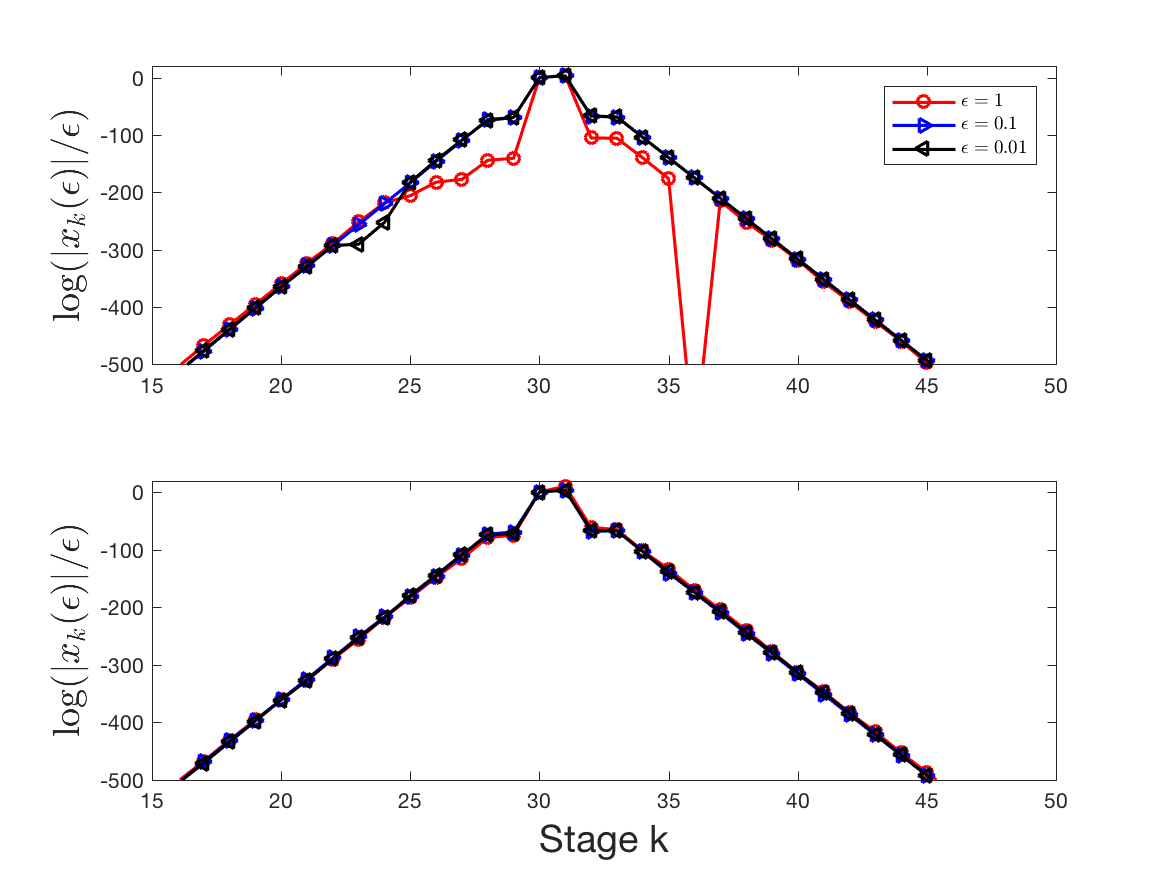}}
	\subfloat[$N=60$, $\mu = (50,10)$]{\label{w8}\includegraphics[width=4.7cm,height=4.8cm]{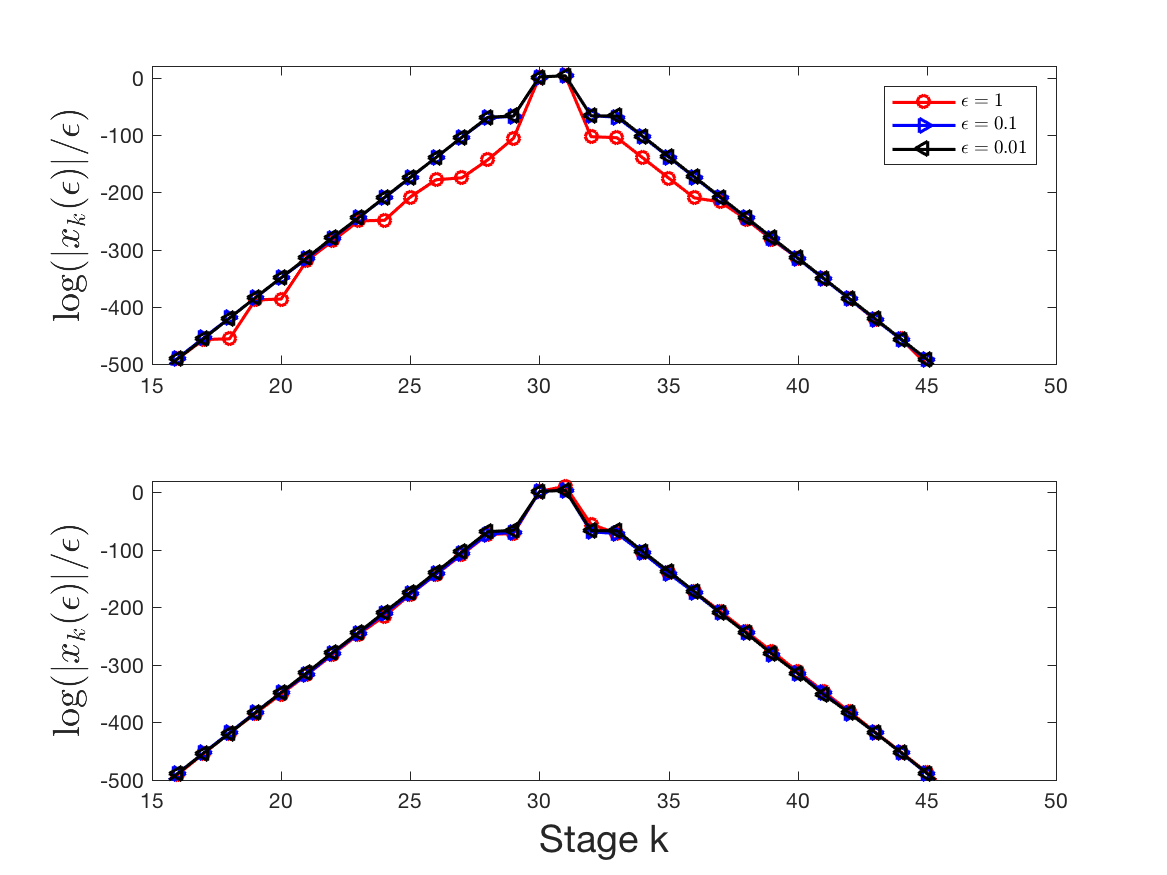}}
	\subfloat[$N=60$, $\mu = (100,15)$]{\label{w9}\includegraphics[width=4.7cm,height=4.8cm]{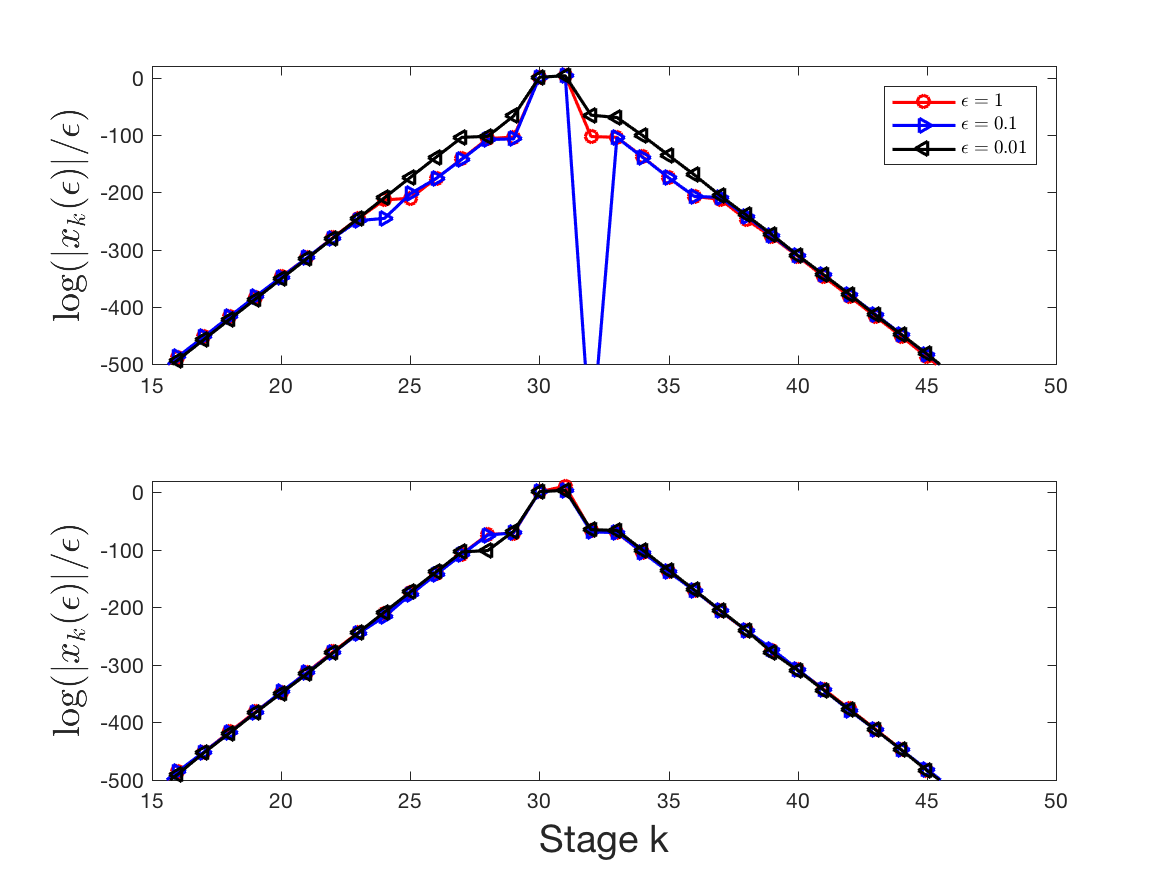}}
\caption{\textit{Sensitivity Plot. Each plot corresponds to a combination of $N$ and $\mu = (\mu_1, \mu_2)$. Within each plot, the upper part corresponds to $f(x)=x$ while the bottom part corresponds to $f(x) = \exp(x)-1$. We truncate the $y$-axis from below -500. Except for some fluctuations, we do observe obvious linear decay trends in all settings, which is consistent with our theoretical result.}}\label{fig:1}
\end{figure}

\section{Conclusion and future work}\label{sec:8}

In this paper, we study the sensitivity analysis for nonlinear dynamic programming with equality constraints. From the theoretical foundation for sensitivity analysis in \cite{Bonnans2000Perturbation} we derive a QDP whose solution is the sensitivity vector but for which the quadratic objective matrix may be indefinite. Dealing with indefinite quadratic matrices has intrinsic difficulties in both practice and theory, which we address through the convexification procedure in \cite{Verschueren2017Sparsity}. To fit that procedure in our setup, we extend the procedure to a QDP with linear shift, and we delve deeper into this algorithm and propose a feasible interval for the regularization hyperparameter $\delta$. Then we show that the algorithm can preserve the boundedness of the data as long as controllability holds, based on which we prove the exponential decay property of sensitivity. That is, we show that under  mild Assumptions \ref{ass:2} and \ref{ass:3}, when we perturb the reference variable at stage $i$ along any direction, the magnitude of the directional derivative for the state and the control variable at stage $k$ (i.e. $\bp_k$ and $\bq_k$) will have exponential decay in terms of $|i-k|$. We also study the dependence of the decay rate with the minimum eigenvalue of the reduced Hessian matrix. We prove that as $\gamma$ gets larger, our bounds modify to the ones of faster decay. Throughout our analysis, we assume a class of NLDP problems, which satisfy uniform SOSC: $\gamma = \lambda_{\min}(Z^THZ)>0$ independent of $N$; and we discuss the likelihood of this setup. Because of this uniform SOSC, all of the quantities in  the theorems including decay rate are independent of $N$. By analogy with the importance of the exponential decay of the sensitivity in our previous work on convex QDP \cite{Xu2018Exponentially}, we anticipate that this result will allow us prove the exponential convergence of temporal decomposition techniques for solving long-horizon NLDP, as was already demonstrated on an instance in \cite{shin2019parallel}, as well as allowing us to prove exponential convergence of nonlinear model predictive control, an extension of \cite{Xu2017Exponentially}.

But we also point out that the analytical setup here may be useful in a general setting in the following sense. Our technique of proof relies on the observation that the sensitivity of the parameter-to-solution mapping of an NLDP, $(\bx^*(\bd), \bu^*(\bd))$, and thus, its approximate local behavior, is the same as of a nonconvex QDP \eqref{equ:9}. Under data boundedness, uniform SOSC, and uniform controllability, we proved that this nonconvex QDP is equivalent (in the sense of having identical parameter-to-solution mappings, which is a feature of this class of convexification approaches \cite{Verschueren2017Sparsity}, stemming from Lemma \ref{lem:3}) to a convex QDP built by the convexification Algorithm \ref{alg:1}  whose data are now uniformly bounded above and objective matrices are bounded below. Therefore classical Riccati techniques can be applied to produce approximate solutions of the NLDP even if the horizon extends to infinity and bring with them stability and robustness-to-uncertainty features. This closes an important gap between the optimization point of view and control point of view of NLDP and may benefit other extensions and design objectives, such as the treatment of uncertainty. 


\section*{Acknowledgment}
This material was based upon work
supported by the U.S. Department of Energy, Office of Science,
Office of Advanced Scientific Computing Research (ASCR) under
Contract DE-AC02-06CH11347 and by NSF
through award CNS-1545046. We gratefully acknowledge discussions with the authors of \cite{Verschueren2017Sparsity}.

\bibliographystyle{siamplain}
\bibliography{ref}
	
\vspace{0.1cm}
\begin{flushright}
	\scriptsize \framebox{\parbox{2.5in}{Government License: The
			submitted manuscript has been created by UChicago Argonne,
			LLC, Operator of Argonne National Laboratory (``Argonne").
			Argonne, a U.S. Department of Energy Office of Science
			laboratory, is operated under Contract
			No. DE-AC02-06CH11357.  The U.S. Government retains for
			itself, and others acting on its behalf, a paid-up
			nonexclusive, irrevocable worldwide license in said
			article to reproduce, prepare derivative works, distribute
			copies to the public, and perform publicly and display
			publicly, by or on behalf of the Government. The Department of Energy will provide public access to these results of federally sponsored research in accordance with the DOE Public Access Plan. http://energy.gov/downloads/doe-public-access-plan. }}
	\normalsize
\end{flushright}	
	
\end{document}